\documentclass[a4paper]{amsart}
\pdfoutput=1

\usepackage[utf8]{inputenc}
\usepackage[T1]{fontenc}
\usepackage{lmodern}
\usepackage[english]{babel}
\usepackage{amsthm, amssymb, amsmath, amsfonts, mathrsfs, mathtools}
\usepackage{microtype}
\usepackage[colorlinks=true,urlcolor=blue,linkcolor=blue,citecolor=blue]{hyperref}
\usepackage{MnSymbol}
\usepackage{scalerel}
\usepackage{xcolor}
\usepackage{accents}
\usepackage{cancel}

\usepackage{bbm}
\definecolor{labelkey}{gray}{.8}
\definecolor{refkey}{gray}{.8}
\definecolor{darkred}{rgb}{0.9,0.1,0.1}
\definecolor{darkgreen}{rgb}{0,0.5,0}

\newtheorem{theorem}{Theorem}[section]
\newtheorem{lemma}[theorem]{Lemma}
\newtheorem{definition}[theorem]{Definition}
\newtheorem{corollary}[theorem]{Corollary}
\newtheorem{proposition}[theorem]{Proposition}

\theoremstyle{remark}
\newtheorem{remark}[theorem]{Remark}
\renewcommand{\qed}{\unskip\nobreak\quad\qedsymbol}

\numberwithin{equation}{section}
\newcommand{\R}{\mathbb{R}}
\newcommand{\CC}{\mathbb{C}}
\newcommand{\Z}{\mathbb{Z}}

\newcommand{\E}{\mathbb{E}}

\newcommand{\F}{\mathcal{F}}

\newcommand{\B}{\mathcal{B}}

\newcommand{\bbT}{\mathbb{T}}

\newcommand{\D}{\mathcal{D}}

\newcommand{\U}{\mathcal{U}}

\def\les{\lesssim}

\DeclareMathOperator{\Var}{Var}

\newcommand{\EE}{\mathbf{E}}
\newcommand{\1}{\mathbbm{1}}

\newcommand{\bT}{\mathbb{T}}

\newcommand{\cZ}{\mathcal{Z}}

\newcommand{\Er}{\mathbb{E}}

\newcommand{\Prm}{\mathbb{P}}
\newcommand{\st}{\;:\;}

\newcommand{\tk}[1]{{\color{blue}#1}}

\global\long\def\Law{\operatorname{Law}}
\global\long\def\sgn{\operatorname{sgn}}
\newcommand{\HH}{\mathbb{H}}
\global\long\def\Re{\operatorname{Re}}
\global\long\def\Im{\operatorname{Im}}
\usepackage{tikz,pgfplots}
\usetikzlibrary{backgrounds,intersections}

\begin{document}

\title[Fluctuation exponents of the KPZ equation on a large torus]{Fluctuation exponents of the KPZ equation\\ on a large torus}
\author{Alexander Dunlap\and Yu Gu\and Tomasz Komorowski}
\address[Alexander Dunlap]{Department of Mathematics, Courant Institute of Mathematical Sciences, New York University, New York, NY 10012 USA}
\address[Yu Gu]{Department of Mathematics, University of Maryland, College Park, MD, 20742 USA}
\address[Tomasz Komorowski]{Institute of Mathematics, Polish Academy of Sciences, ul.
	Śniadeckich 8, 00-636 Warsaw, Poland}

\begin{abstract}
  We study the one-dimensional KPZ equation on a large torus, started at equilibrium. The main results are optimal variance bounds in the super-relaxation regime and part of the relaxation regime.
	\bigskip


	\medskip

\end{abstract}
\maketitle
\section{Introduction}

\subsection{Main result}
We consider the stochastic heat equation (SHE) on a one-dimensional torus of size $L$:
\begin{equation}\label{e.she}
	\partial_t \U=\frac12\Delta \U+\U \eta(t,x), \quad\quad x \in \mathbb{T}_L=\R/L\Z,
\end{equation}
where $\eta$ is a spacetime white noise, and we will identify $\mathbb{T}_L$ with the interval $[0,L]$ in the usual way. We assume the initial data takes the form
\begin{equation}
	\U(0,x)=e^{B(x)},
\end{equation}
with $B$  a standard Brownian bridge on $[0,L]$ with $B(0)=B(L)=0$.
Thus, the Cole--Hopf solution $h(t,x)=\log \U(t,x)$ to the corresponding Kardar-Parisi-Zhang (KPZ)
equation starts from $\log \U(0,x)=B(x)$, and is stationary under the evolution
\[
 \partial_t h(t,x) = \frac12\Delta h(t,x)+\frac12(\partial_x h(t,x))^2 + \eta(t,x)
\]
up to spatially constant height shifts. 
The goal of this paper is to study the asymptotics of the variance
$\Var[h(t,x)]$, as $t$ and $L$ go to
infinity together.
We write $a\les b$ if $a\leq Cb$ for some constant $C>0$ independent of $t \geq1$ and $x\in [0,L]$, and $a\asymp b$ if $a\les b$ and $b\les a$.
Here is the main result:
\begin{theorem}\label{t.mainth}
	 Suppose that $L=\lambda t^{\alpha}$ for some
         $\lambda>0$ and $\alpha\geq 0$. Then, there is a universal constant $\delta>0$ so that the following
        holds:
	\begin{equation}\label{e.conjecture1}
		\Var [h(t,x)] \asymp
		\begin{cases}
			t^{1-\frac{\alpha}{2}}, & \text{if }\alpha\in[0,2/3)\text{ and }\lambda\in(0,\infty);                             \\
			t^{2/3},        & \text{if }\alpha = 2/3\text{ and }\lambda\in(0,\delta).
		\end{cases}
              \end{equation}
In the general case of $\alpha=2/3$ and $\lambda\in(0,\infty)$,  the upper bound holds 
\[
\Var[ h(t,x)]\les  t^{2/3}.
\] 
%
\end{theorem}

 For all values of $\alpha\geq 2/3$ and $\lambda>0$, we expect that $\Var [h(t,0)]\asymp t^{2/3}$. The extreme case of $\alpha=\infty$, i.e., when the equation is posed on the whole line, was first proved in \cite{BQS11}. Our approach only allows to obtain the fluctuations of size $t^{2/3}$ in the case of $\alpha=2/3$ and $\lambda$ sufficiently small. We refer to Remark~\ref{r.sub} below for more discussions in this regard.

\subsection{Context}

The study of the $1+1$ KPZ universality class has witnessed  important progress in recent years, with the construction of the fixed point and the proofs of various models converging to it under appropriate scalings. There is a huge body of literature on the subject, and we refer to the results e.g.\ in \cite{ACQ11,SS10a,SS10b,BCF14,BCFV15,DOV21,MQR20,QS20,Vir20}, the reviews \cite{Qua12,Cor12,QS15} and the references therein. In particular, for the KPZ equation on the whole line, it was shown in \cite{QS20,Vir20} that, under the $1:2:3$ scaling, the centered height function converges to the KPZ fixed point. Nevertheless, most  studies have been focusing on exact solvable models and relying on exploiting certain integrable structures. So it still remains a major open problem to derive the correct scaling exponents for general models     and  to further prove the convergence of the rescaled random fluctuations.

In this paper, we study the KPZ equation posed on a torus of size $L$,
started at equilibrium, and investigate the growth of the
variance of the height function as $t\to\infty$. By choosing $L\sim
t^{\alpha}$ and tuning the parameter $\alpha$, the model interpolates
two universality classes: the Gaussian class ($\alpha=0$, $L\sim
O(1)$) and the KPZ class ($\alpha=\infty$,
$L=\infty$). The critical length scale is $L\sim
t^{2/3}$, which is the so-called relaxation regime where the heights
at all points on the torus happen to be  correlated.  The regimes of
$\alpha<2/3$ and $\alpha>2/3$ are the so-called super- and
sub-relaxation regimes, respectively. For $\alpha=0$, which is the
case of a fixed torus, the Gaussian fluctuations were proved in \cite{GK21}; for $\alpha=\infty$, which is the case of  the whole line, the $t^{1/3}$ fluctuations were obtained in \cite{BQS11} and the weak convergence of the rescaled fluctuations were further derived in \cite{BCFV15} in the stationary regime. Here we devise a new approach which does not rely on integrable tools, through which we obtain the  correct fluctuation exponents in the super-relaxation  regime and a part of the relaxation regime.

The idea of our approach is easy to describe: for a  torus of fixed
size $L$, the result in \cite{GK21} shows the diffusive scaling of
$\Var h(t,0)\sim \sigma_L^2t$, with some 
diffusion constant $\sigma_L^2$. It is natural to expect that, in order to obtain the sub-diffusive behavior in the case of $L=\infty$, the
diffusion constant $\sigma_L^2$ must vanish as $L\to\infty$. While the
expression of the diffusion constant is model-dependent, its decay
rate  is expected to be universal, and it is related to the
characteristic exponents of the  universality class. The correct rate
 $\sigma_L^2\sim L^{-1/2}$ was proved in
\cite{DEM93} for TASEP on the torus, a different model in the
universality class. In our setting, we will prove that $\sigma_L^2\sim
L^{-1/2}$ in Theorem~\ref{prop:sigmaprop} below. Roughly speaking, the fluctuations of the height function can be decomposed into two parts, the ``absolute fluctuations'' which we can described by $\Var h(t,0) \sim t\sigma_L^2 $, and the transversal roughness, i.e., the fluctuations of $h(t,\cdot)-h(t,0)$, which is
    of order $L$ in our case, due to the fact that the invariant measure for $h(t,\cdot)-h(t,0)$ is given by the law of the Brownian bridge, whose variance is of order $L$ on the torus of size $L$. With the decay of  $\sigma_L^2\sim L^{-1/2}$, we need $L\sim t^{2/3}$ to balance  $t\sigma_L^2$ and $L$, which further implies the variance of the height function is $t\sigma_L^2\sim t^{2/3}$. This picture was sketched in \cite{DEM93} for TASEP, and the core of our argument is to make it rigorous for the KPZ equation.

For TASEP on the torus, more precise results have been derived
in \cite{BL16,BL18,Liu18,BL19,BLS20,BL21}. The rescaled random fluctuations
are described explicitly in the relaxation regime, which has
  been further shown to interpolate the Gaussian statistics and the
Tracy-Widom type statistics. Exact formulas of the transition probability for ASEP on the torus were derived in \cite{LSW20}. For models of first passage percolation \cite{CD13} and longest increasing path \cite{DJP18}, similar variance bounds as in \eqref{e.conjecture1} have been derived in the corresponding super-relaxation regimes and the Gaussian fluctuations were also proved. Another related work is \cite{CMR19}, where the authors considered directed polymers in a random environment on the complete graph of size $N$, and by exploiting  the connection to products of random matrices, they performed large $N$ asymptotics when the disorder variables are sampled from a certain special distribution.  

A key part of our argument relies on the fact that the law of the
Brownian bridge modulo an additive constant is
the unique invariant measure for the KPZ equation  on the torus; see
\cite{BG97,FQ15,GP17,HM18}. The fact that the invariant measure is
explicit is crucial in our analysis, especially for the proof of the
decay of $\sigma_L^2\sim L^{-1/2}$. For the KPZ equation or the stochastic Burgers equation, on the torus with a noise that is white in time and smooth in space, the existence and uniqueness of the invariant measure is well-known (see e.g.\ \cite{Sin91,Ros21,GK21}), but there is very little information regarding the quantitative properties of the invariant measure. Part of our proof applies to the case of a colored noise in all dimensions, and our analysis reduces the proof of results of the form \eqref{e.conjecture1} to the study of two explicit functionals of the invariant measure; see the discussion in Section~\ref{s.highd}. On this point, we mention the recent works on deriving the explicit formulas for the invariant measure for the KPZ equation posed on an interval with Neumann boundary conditions \cite{CK21,BKWW21,BD21}.

Our analysis is restricted to the torus, in which case we can define
the endpoint distribution of the directed polymer  at stationarity and
use it as the underlying driving Markov process. In terms of the
solution to SHE, the endpoint distribution takes the form
$\rho(t,x)=\U(t,x)/\int_0^L \U(t,x')dx'$, which is not well-defined on
the whole line when we start the KPZ from a two-sided Brownian
motion. In many studies of the models in the $1+1$ KPZ universality
class, the underlying  process is chosen to be the solution of the
Burgers equation, or some discretization of the Burgers equation, so
that there is always  a well-defined stationary measure. For the
problem on the torus, we find the Markov process
$\{\rho(t,\cdot)\}_{t\geq0}$ useful, since the centered height
function can be written as a stochastic integral of an explicit
functional of $\rho$ with respect to the underlying noise; see
Proposition~\ref{p.co} below. On the whole line, when $\U$ starts from
some $L^1$ initial data, the Markov process
$\{\rho(t,\cdot)\}_{t\geq0}$ is also well-defined, but it does not
have a invariant probability measure, as the polymer endpoint spreads out as $t\to\infty$. Nevertheless, by embedding the endpoint distribution into
an abstract space, which factors out the spatial shift, the localization properties of the endpoint distribution have been studied in \cite{BC20,BM19,BS20}.

\subsection{Ingredients of the proof}

For a torus of fixed size $L$, the following central limit theorem was proved in \cite{GK21}: there exist $\gamma_L,\sigma_L^2>0$ such that for any $x\in \bT_L$,
\begin{equation}\label{e.clt}
	\frac{\log \U(t,x)+\gamma_L t}{\sqrt{t}}\Rightarrow N(0,\sigma_L^2), \quad\quad \text{ as } t\to\infty.
\end{equation}
The variance $\sigma_L^2$ was derived in \cite{CK21} using a homogenization argument. The argument
is based on a semimartingale decomposition of the free energy of the   directed polymer, the construction of the corrector by solving a Poisson equation, and an application of the martingale central limit theorem. As a result, the effective diffusion constant $\sigma_L^2$ depends on the corrector, and is expressed in terms of a Green--Kubo-type formula. This is not surprising, since the variance typically  depends on the temporal correlations of the underlying process.

As mentioned previously, to study the size of the fluctuations for
large $L$, the key is to derive the decay rate of the diffusion
constant $\sigma_L^2$ as $L\to\infty$. However, the formula in
\cite{GK21} for $\sigma_L^2$ involves the corrector, hence is not
amenable to asymptotic analysis as $L\to\infty$. So the first task is
to derive a more explicit expression for
$\sigma_L^2$. Somewhat  surprisingly, it can be written as  the
expectation of an explicit exponential functional of  correlated
Brownian bridges; see \eqref{e.defsigma} below. Our approach is based
on an application of the Clark--Ocone formula, which expresses the
centered height function as an It\^o integral with respect to the underlying
noise. Then by It\^o isometry, the variance of the height function is written as  an explicit functional of the
directed polymer. An immediate consequence of the Clark--Ocone formula is the Gaussian Poincar\'e inequality, which is intimately connected to the noise-sensitivity literature: to study the variance of the height function, one may want to understand how it depends on the underlying noise, and the sensitivity in
this context is naturally described by the Malliavin derivative. While the Gaussian Poincar\'e inequality only yields the
sub-optimal diffusive upper bound in this case, the Clark--Ocone
formula provides an \emph{identity} that we can further analyze. At this point, it is worth pointing out that various possible ways of improving the Gaussian Poincar\'e
inequality to obtain the so-called superconcentration 
has been explained in the excellent monograph
\cite{Cha14}.

The central limit theorem \eqref{e.clt} holds in the regime when $t\to\infty$ with fixed $L$, and so is not suited for our case where $L$ and $t$ go to infinity
simultaneously.  Thus, the second task is to create a type of
``stationarity'' and come up with an approximation of $\log
\,\U(t,x)$, so that we can compute the variance of the approximation
for all finite $t$ and $L$. This is the reason we start the equation
at equilibrium. The approximation we choose is of the form  $\log
\int_0^L \U(t,x)e^{W(x)} dx$, where $W$ is another standard Brownian
bridge on $[0,L]$, independent of the noise $\eta$ and the initial
data $B$. One should view the factor of $e^{W(\cdot)}$, after a
normalization, as the prescribed endpoint distribution of the directed
polymer of length $t$, so that the integral $\int_0^L \U(t,x)
e^{W(x)}dx$ is related to the partition function of the polymer which
not only starts at equilibrium but also ends at an independent equilibrium. By further using the time-reversal symmetry of the Green's function of the SHE in the Clark--Ocone representation, we obtain the aforementioned ``stationarity.'' 

The final step is to prove the decay rate $\sigma_L^2 \sim L^{-1/2}$. Although the diffusion constant obtained in \eqref{e.defsigma} below is the expectation of an explicit exponential functional of Brownian bridges, the analysis of this functional is nontrivial and constitutes the most technical part of the paper. The difficulty comes from the fact that there are two correlated Brownian bridges,  the correlation ultimately coming from the conditional expectation taken in the Clark-Ocone formula. The decay rate of $L^{-1/2}$ is intimately connected to the correlation. If the two Brownian bridges were independent or identical, then we could use Yor's density formula in \cite{Yor92} to compute $\sigma_L^2$ explicitly, yielding, respectively, $L^{-1}$ decay or no decay at all. The nontrivial correlation interpolates between these scenarios and leads to the correct exponent $-1/2$.

Our analysis of $\sigma_L^2$ proceeds by first writing the correlated Brownian bridges as a linear transformation of two independent Brownian bridges. It turns out that the contributions to the functional in \eqref{e.defsigma} mainly come from the event that the correlated Brownian bridges do not travel too far above zero, which corresponds to the event that the independent Brownian bridges do not stray too far outside of a wedge in $\R^2$. The angle of the wedge ($2\pi/3$) comes from the correlation, and is what leads to the exponent $-1/2$. While the natural approach to the notion of ``too far outside of the wedge'' would seem to involve Brownian motion with a ``soft'' cutoff aligned with the sides of the wedge, we do not at present  know how to analyze this directly. (See the discussion in Section~\ref{subsec:LQM} below.) Instead, for the upper bound, we approximate the ``soft'' cutoff by a sequence of ``hard'' cutoffs and use a stopping time argument. For the lower bound, we approximate the ``soft'' cutoff by a single ``hard'' cutoff, and use an entropic repulsion argument similar to that of \cite{Bra83,CHL19supp}. These approximations allow us to use the well-known formula for the heat kernel of a standard planar Brownian motion killed upon exiting a wedge (that is, with a ``hard'' cutoff) to undergird the analysis. In particular, this formula is ultimately the source of the correct exponent.

\subsection{Organization of the paper}
The rest of the paper is organized as follows. In Section~\ref{s.variance}, we derive an explicit expression of the diffusion constant $\sigma_L^2$. In Section~\ref{s.decomposition}, we prove a decomposition of the centered height function into two parts, one corresponding to the diffusive fluctuation of order $t\sigma_L^2$, and the other corresponding to the transversal roughness of   order $L$. Then the proof of Theorem~\ref{t.mainth} is completed by balancing these two terms. Section~\ref{sec.sigmaasympt} is devoted to the proof of the decay $\sigma_L^2\sim L^{-1/2}$. In Section~\ref{s.discussion}, we discuss several possible extensions, including another proof of the central limit theorem in \eqref{e.clt}, a possible connection to Liouville quantum mechanics which could be related to computing $\sigma_L^2$ exactly, and the case of the colored noise and the equation in high dimensions.

\subsection*{Acknowledgements} A.D. was partially supported by the NSF Mathematical Sciences Postdoctoral Research Fellowship through grant number DMS-2002118.
Y.G. was partially supported by the NSF through DMS-2203007/2203014. T.K. acknowledges the support of NCN grant 2020/37/B/ST1/00426. We would like to thank Zhipeng Liu, Mateusz Kwa\'snicki, Stefano Olla and Li-Cheng Tsai for discussions. We thank the anonymous referee for several helpful suggestions and comments which helped to improve the presentation of the paper. 

\section{A formula for the diffusion constant}
\label{s.variance}

In this section, we derive a formula for the variance appearing in \eqref{e.clt}. The result is based on an application of the Clark--Ocone formula.

We first introduce some notations. Let $ C_+(\bT_L)$ be the subset of  $ C(\bT_L)$
-- the space of continuous functions on $\bT_L$ -- consisting of all
non-negative  functions
that are not identically equal to $0$.  For any    $f,g\in C_+(\bT_L)$, let
$\U(t,x;g)$ denote the solution to \eqref{e.she} starting from  $g$, i.e.\ with $\U(0,x;g)=g(x)$, and let
\begin{equation}\label{e.defX}
	X_{f,g}(t)\coloneqq\log \int_0^L \U(t,x;g)f(x)dx.
\end{equation}
Define $\cZ(t,x;s,y)$ to be the propagator of the SHE from $(s,y)$ to $(t,x)$, i.e.,
\begin{equation}\label{e.propa}
	\begin{aligned}
		 & \partial_t \cZ(t,x;s,y)=\frac12\Delta_x \cZ(t,x;s,y)+\cZ(t,x;s,y) \eta(t,x), \quad\quad t>s, \\
		 & \cZ(s,x;s,y)=\delta(x-y).
	\end{aligned}
\end{equation}
Then we can write
\[
	X_{f,g}(t)=\log \int_0^L\int_0^L \cZ(t,x;0,y)g(y)f(x)dxdy.
\]
Let $\mathbb D^\infty(\bbT_L)$ be the set of all
Borel probability measures on $\bbT_L$ with continuous densities with respect to
the Lebesgue measure. In what follows we shall identify a measure
from $\mathbb D^\infty(\bbT_L)$ with its respective density.
If $f,g\in \mathbb D^\infty(\bbT_L)$  we can view   $\int_0^L\int_0^L \cZ(t,x;0,y)g(y)f(x)dxdy$ as the partition function of the point-to-point directed polymer of length $t$, with the two endpoints sampled from $f,g$ respectively. Then $X_{f,g}(t)$ is the corresponding free energy.

For any $g\in C_+(\bT_L)$, define
\begin{equation}
	\rho(t,x;g)=\frac{\U(t,x;g)}{\int_0^L \U(t,x';g)dx'},\label{eq:rhodef}
\end{equation}
which is  the endpoint distribution of the directed polymer with the
initial point sampled from the density $g(\cdot)/\int_0^L
	g(x')dx'$ on $[0,L]$. In particular, \eqref{eq:rhodef} with $t=0$ means that $$
	\rho(0,x;g)=\frac{g(x)}{\int_0^L g(x')dx'}.
$$
From \cite[Theorem 2.3]{GK21}, we know that as  a Markov process,
$\{\rho(t)\}_{t\geq0}$ has a unique  invariant probability
measure $\pi_\infty$ on $\mathbb D^\infty(\bbT_L)$. Thanks to
\cite[Theorems 2.1 (part (1)) and 1.1]{FQ15}, the invariant measure  $\pi_\infty$ is the law   of the random element
\[
	\frac{e^{B(\cdot)}}{\int_0^L e^{B(x')}dx'},
\]
where $B$ is a Brownian bridge on $[0,L]$ with $B(0)=B(L)=0$. Note that $\pi_\infty$ is invariant in law under periodic rotations of the torus $\bT_L$, even though the law of the Brownian bridge $B$ is not.
We also define the probability measure $\mu$ on $C_+(\bT_L)$ to be the distribution of $e^{B(\cdot)}$, and later we will sample $f,g$, which appeared in \eqref{e.defX}, independently from $\mu$.

Since there are several sources of randomness in our setup, to avoid
confusion, in Sections~\ref{s.variance} and \ref{s.decomposition}, we
use $\EE$ to denote the expectation \emph{only} with respect to the noise $\eta$. The expectations with respect to $f$, $g$, and the Brownian bridges $B$ and $W$ will be  denoted by $\E_f$, $\E_g$, $\E_B$, and $\E_W$, respectively. When there is no possibility of confusion, we use $\E$ as the expectation with respect to all possible randomnesses.

For any $t>s\geq0$, $f,h\in C_+(\bT_L)$ and $y\in\bT_L$, define the deterministic functional
\begin{equation}\label{e.defG}
	\mathscr{G}_{t,s}(f,h,y)\coloneqq\EE\bigg[\frac{\int_0^L f(x) \cZ(t,x;s,y)dx }{\int_0^L\int_0^L  f(x)\cZ(t,x;s,y')h(y')dxdy'}\bigg].
\end{equation}

We have the following proposition which comes from a direct application of the Clark--Ocone formula.
\begin{proposition}\label{p.co}
	Fix any $t,L>0$ and $f,g\in C_+(\bT_L)$, we have
	\begin{equation}\label{e.co}
		X_{f,g}(t)-\EE X_{f,g}(t)= \int_0^t\int_0^L \mathscr{G}_{t,s}(f,\rho(s;g),y)\rho(s,y;g)\eta(s,y)dyds.
	\end{equation}
\end{proposition}
We postpone the proof of Proposition~\ref{p.co} until Section
\ref{sec3.1} and present first some corollaries of the result.

As mentioned already, one can view $X_{f,g}(t)$ as the free energy of the directed polymer of length $t$, with the two endpoints sampled from $f(\cdot)/\int_0^L
	f(x')dx'$ and $g(\cdot)/\int_0^L
	g(x')dx'$, respectively.
The idea is to sample $f,g$ from $\mu$ so that the start- and end-point distributions of the polymer path are both stationary but independent of each other and the noise.

First, when we sample $f$ from $\mu$ (in symbols, when $f\sim \mu$),
independent of the noise $\eta$, the integrand on the r.h.s.\ of
\eqref{e.co} can be simplified. We use $\E_f$ to denote the
expectation with respect to $f\sim \mu$. In this case it is
straightforward to check, for any fixed $t>s$, that if we define the random density \begin{equation}\label{e.defrhof}
	\widetilde{\rho}_{t,f}(s,y)\coloneqq\frac{\int_0^L f(x)\cZ(t,x;s,y)dx}{\int_0^L \int_0^L f(x)\cZ(t,x;s,y')dxdy'},
\end{equation}
then $\widetilde{\rho}_{t,f}(s,\cdot)\sim\pi_\infty$. Indeed, $\widetilde{\rho}_{t,f}(s,\cdot)$ can be
viewed as the endpoint distribution of the ``backward'' polymer of
length $t-s$ with its starting point distributed according to $f(\cdot)/\int_0^L f(x)dx$. In particular, this means that
\[
 \widetilde{\rho}_{t,f}(s,\cdot)\overset{\text{law}}= \frac{e^{B(\cdot)}}{\int_0^L e^{B(x')}dx'}.
\]
So for any $h\in C_+(\bT_L)$ and $y\in \bT_L$, we have
\begin{equation}\label{e.EfG}
	\begin{aligned}
		\E_f\mathscr{G}_{t,s}(f,h,y) &= \E_f\EE\bigg[\frac{\int_0^L f(x) \cZ(t,x;s,y)dx }{\int_0^L\int_0^L  f(x)\cZ(t,x;s,y')h(y')dxdy'}\bigg] \\
		                             & =\E_f\EE\bigg[\frac{\widetilde{\rho}_{t,f}(s,y) }{\int_0^L \widetilde{\rho}_{t,f}(s,y')h(y')dy'}\bigg]  \\
		                            & = \E_B \bigg[\frac{e^{B(y)} }{\int_0^L e^{B(y')}h(y')dy'}\bigg]\eqqcolon\mathcal{A}(h,y).
	\end{aligned}
\end{equation}
Note that the functional $\mathcal{A}$ is deterministic.

A direct corollary of \eqref{e.co} and \eqref{e.EfG} is the following.
\begin{corollary}
	For any $g\in  C_+(\bT_L)$ and $t,L>0$, we have
	\begin{equation}\label{e.co0}
		\E_f X_{f,g}(t)-\E_f\EE X_{f,g}(t)=\int_0^t\int_0^L
		\mathcal{A}(\rho(s;g),y)\rho(s,y;g)\eta(s,y)dyds.\end{equation}
\end{corollary}

Furthermore, when we sample $g\sim \mu$ independently,
$\rho(s,\cdot;g)$ is the endpoint distribution of the ``forward''
polymer of length $s$, starting from the stationary distribution, so the variance of \eqref{e.co0} can be computed explicitly. We use $\E_g$ to denote the expectation with respect to $g\sim \mu$ which is independent of $f$ and $\eta$.
\begin{corollary}\label{c.co1}
	For any $t,L>0$, we have \begin{equation}\label{e.co2}
	\begin{aligned}
		&\E_g\EE\Big[ \big(\E_f X_{f,g}(t)-\E_f\EE X_{f,g}(t)\big)^2\Big]\\
		&=\int_0^t\int_0^L \E_g\EE \big[ \mathcal{A}(\rho(s;g),y)\rho(s,y;g)\big]^2 dy     ds = t\sigma_L^2,
		\end{aligned}
	\end{equation}
	with \begin{equation}\label{e.defsigma}
			\sigma_L^2 =\E_{B} \frac{\int_0^L e^{B_1(x)+B_2(x)+2B_3(x)}dx}{\int_0^L e^{B_1(x)+B_3(x)}dx\int_0^L e^{B_2(x)+B_3(x)}dx},
	\end{equation}
	where $B_1,B_2,B_3$ are independent Brownian bridges on $[0,L]$ with
	$B_j(0)=B_j(L)=0$, $j=1,2,3$.
\end{corollary}
\begin{proof}
	For any deterministic $g$, applying the It\^o isometry to \eqref{e.co0}, we obtain
	\[
		\EE\Big[ \big(\E_f X_{f,g}(t)-\E_f\EE X_{f,g}(t)\big)^2\Big]=\int_0^t\int_0^L \EE\Big[\big(\mathcal{A}(\rho(s;g),y)\rho(s,y;g)\big)^2\Big] dyds.
	\]
	Further taking the expectation with respect to $g\sim \mu$ we conclude the
	first equality in \eqref{e.co2}.  To prove the second one, we use the
	fact that   $\rho(s,\cdot;g)\sim \pi_\infty$ for any $s$.
	Then, using \eqref{e.EfG}, we obtain
	\begin{align*}
	  \E_g&\EE\Big[\big(\mathcal{A}(\rho(s;g),y)\rho(s,y;g)\big)^2\Big] \\
	  &= \E_B\left[ \mathcal{A}\left(e^{B_3(\cdot)}/\int_0^L e^{B_3(x)}dx,y\right) ^2\left( e^{B_3(y)}/\int_0^L e^{B_3(y')}dy'\right)^2\right]\\
	  &=  \E_B\left[\frac{e^{B_1(y)+B_2(y)+2B_3(y)}}{\prod_{j=1}^2\int_0^L \left(e^{B_j(y')+B_3(y')}/\int_0^L e^{B_3(x)}dx\right)dy'}\left( \int_0^L e^{B_3(y')}dy'\right)^{-2}
	  \right],
	\end{align*}
and thus \eqref{e.defsigma}.
\end{proof}

\subsection{Proof of Proposition~\ref{p.co}}

\label{sec3.1}


Denote $\D$ the Malliavin derivative with respect to $\eta$, using
the chain rule, we derive
\begin{equation}
	\label{012509-21}
	\begin{aligned}
		\D_{s,y}X_{f,g}(t) & =\D_{s,y}\log \int_0^L \U(t,x;g)f(x)dx                              \\
		                   & =\frac{\int_0^L \D_{s,y}\U(t,x;g)f(x)dx}{\int_0^L \U(t,x;g)f(x)dx}.
	\end{aligned}
\end{equation}
The Malliavin derivative of the SHE can be found, for example, in the proof of
\cite[Theorem 3.2]{CKNP20}. It is given by
\[
	\D_{s,y}\U(t,x;g)=\cZ(t,x;s,y) \U(s,y;g),
\]
where we recall that $\cZ$ is the propagator of the SHE, defined in \eqref{e.propa}. Then from the fact that
\[
	\U(t,x;g)=\int_0^L \cZ(t,x;s,y')\U(s,y';g)dy',\quad t>s,
\]
we obtain, after substituting into  \eqref{012509-21},
\begin{equation}\label{e.DsyX}
	\begin{aligned}
		\D_{s,y}X_{f,g}(t) & =\frac{\int_0^L f(x) \cZ(t,x;s,y)dx }{\int_0^L\int_0^L f(x)\cZ(t,x;s,y')\U(s,y';g)dxdy'}\times \U(s,y;g) \\
		                   &= \frac{\int_0^L f(x) \cZ(t,x;s,y)dx }{\int_0^L f(x)\cZ(t,x;s,y')\rho(s,y';g)dxdy'}\times \rho(s,y;g)\\
		                   &=\frac{\widetilde{\rho}_{t,f}(s,y)\rho(s,y;g)}{\int_0^L\widetilde{\rho}_{t,f}(s,y')\rho(s,y';g)dy'},
	\end{aligned}
\end{equation}
where $\widetilde{\rho}_{t,f}$ was defined in \eqref{e.defrhof}. The proof of the square integrability of the Malliavin derivative is left in Appendix~\ref{s.sqin}. Applying the Clark--Ocone formula (see e.g.\ \cite[Proposition 6.3]{CKNP19}), we have
\begin{equation}\label{e.co11}
	X_{f,g}(t)-\EE X_{f,g}(t)=\int_0^t \int_0^L \EE[ \D_{s,y}X_{f,g}(t)\mid\F_s]\eta(s,y)dyds,
\end{equation}
where 
$\{\F_s\}_{s\geq0}$ is the filtration generated by $\eta$. The conditional expectation in \eqref{e.co11} is then equal to
\[
	\EE[\D_{s,y}X_{f,g}(t)\mid\F_s]=\EE\bigg[\frac{\int_0^L f(x) \cZ(t,x;s,y)dx }{\int_0^L\int_0^L f(x)\cZ(t,x;s,y')\rho(s,y';g)dxdy'}\ \bigg|\ \F_s\bigg] \rho(s,y;g).
\]
Since (for $t\ge s$) $\cZ(t,\cdot;s,\cdot)$ is independent of $\F_s$, the conditional expectation on the r.h.s.\ of the above display is equal to $\mathscr{G}(f,\rho(s;g),y)$, and this   completes the proof.\qed

\section{A decomposition of the centered height function}
\label{s.decomposition}

In the present section, we prove a key decomposition that relates $\log \U(t,0;g)$  to  $X_{f,g}$ defined in \eqref{e.defX}. Here $g(x)=e^{B(x)}$, where $B$ is a Brownian bridge on $[0,L]$ with $B(0)=B(L)=0$. For any $L>0$, define the stationary process
\begin{equation}\label{e.defY}
	Y_L(t)\coloneqq\E_W \log \int_0^L  \frac{\U(t,x;g)}{\U(t,0;g)}e^{W(x)}dx, \quad\quad t\geq0,\end{equation}
where $W$ is a Brownian bridge on $[0,L]$ with $W(0)=W(L)=0$, and is independent of $\eta$ and $B$. The above expectation $\E_W$  is only on $W$, so $Y_L(t)$ retains the randomnesses from $\eta$ and $B$. Here is the main result:

\begin{proposition}\label{p.decom}
	For any $t,L>0$, we have
	\begin{equation}
		\begin{aligned}
			\log \U(t,0;g)-\E_g\EE \log \U(t,0;g)=I_L(t)+J_L(t),
		\end{aligned}
	\end{equation}
	with
	\begin{equation}\label{e.defI}
		I_L(t)\coloneqq\E_f X_{f,g}(t)-\E_f\EE X_{f,g}(t).
	\end{equation}
	and
	\begin{equation}\label{e.defJ}
		J_L(t)\coloneqq Y_L(0)-Y_L(t).
	\end{equation}
\end{proposition}

\begin{proof}
	We start from the following algebraic identity, which follows from the definitions
	\eqref{e.defX} and \eqref{e.defI}:
	\begin{equation}\label{e.9161}
			\log \U(t,0;g)+\E_f \log \int_0^L \frac{\U(t,x;g)}{\U(t,0;g)}f(x)dx
			=I_L(t)+\E_f\EE X_{f,g}(t).
	\end{equation}
	Now we consider each of the expectations in \eqref{e.9161}
        (the one on the left and the one on  the right).
	First, since $\E_f$ is the expectation on $f\sim \mu$, we have
	$$
		\E_f \log \int_0^L \frac{\U(t,x;g)}{\U(t,0;g)}f(x)dx=Y_L(t).
	$$ Secondly,  we claim that
	\begin{equation}\label{e.9162}
		\E_f\EE X_{f,g}(t)=\E_f\EE \log \int_0^L \U(t,x;g)f(x)dx=\E_f\EE \log \int_0^L \U(t,x;f)g(x)dx.
	\end{equation}
	We will show \eqref{e.9162} momentarily. First, we use it to finish the    proof of the proposition.
	From \eqref{e.9162}, we  further derive
	\begin{equation}\label{eq:develope9162}
		\E_f\EE X_{f,g}(t)=\E_f\EE \log \int_0^L \frac{\U(t,x;f)}{\U(t,0;f)}g(x)dx+\E_f\EE \log \U(t,0;f).
	\end{equation}
	Note that with $f\sim \mu$,  we have
	$$
		\frac{\U(t,\cdot;f)}{\U(t,0;f)}\stackrel{\text{law}}{=}e^{W(\cdot)},
	$$
	which implies 
\begin{equation}
	\begin{split}
	\E_f\EE \log \int_0^L
                \frac{\U(t,x;f)}{\U(t,0;f)}g(x)dx
&=\E_W  \log \int_0^Le^{W(x)}g(x)dx\\
&=Y_L(0)+\log g(0)=Y_L(0)
\end{split},\label{eq:getYL0}\end{equation}
	since $g(0)=1$, $\mu$-a.s. Since, obviously,
	\begin{equation}\E_f\EE \log \U(t,0;f)=\E_g\EE \log \U(t,0;g),\label{eq:EfEg}\end{equation}
	combining \eqref{eq:getYL0} and  \eqref{eq:EfEg} in \eqref{eq:develope9162}, we obtain
	\[
		\E_f\EE X_{f,g}(t)=Y_L(0) +\E_g\EE \log \U(t,0;g),\quad
		\mu\text{-a.s. in }g.
	\]
	Substituting for the respective expressions into
        \eqref{e.9161} we get
	\[
		\log \U(t,0;g)+Y_L(t)=I_L(t)+Y_L(0)+\E_g\EE \log \U(t,0;g),
	\]
	which completes the proof.

	It  remains to prove \eqref{e.9162}. The first equality is just the
	definition of $X_{f,g}(t)$. For the second equality, we use the
	propagator of SHE and write
	\[
		\int_0^L \U(t,x;g)f(x)dx=\int_0^L\int_0^L \cZ(t,x;0,y)g(y)f(x)dxdy.
	\]
	Then \eqref{e.9162} is a consequence of the fact that for each fixed $t>0$, we have
	\begin{equation}\label{e.reversal}
		\{ \cZ(t,x;0,y)\}_{x,y\in\bT_L}\stackrel{\text{law}}{=}\{ \cZ(t,y;0,x)\}_{x,y\in\bT_L}.
	\end{equation}
	For the convenience of readers, we provide a proof of \eqref{e.reversal} in the appendix.
\end{proof}

It turns out that the marginal distribution of $Y_L(t)$ is rather explicit, and we have the following lemma:

\begin{lemma}\label{l.varYL}
	For any $t\geq0$, we have as $L\to\infty$,
	\begin{equation}
		\frac{Y_L(t)-\log L}{\sqrt{L}}\Rightarrow \E_{\mathcal{W}}\max_{x\in[0,1]}\{\B(x)+\mathcal{W}(x)\}
	\end{equation}
	in distribution, and
	\begin{equation}
		\frac{1}{L}\Var[Y_L(t)]\to  \Var \big[\E_{\mathcal{W}}\max_{x\in[0,1]}\{\B(x)+\mathcal{W}(x)\}\big]>0.
	\end{equation}
	Here $\B,\mathcal{W}$ are independent Brownian bridges on $[0,1]$
	satisfying  $\B(0)=\B(1)=\mathcal{W}(0)=\mathcal{W}(1)=0$.
\end{lemma}

\begin{proof}
	With $g(x)=e^{B(x)}$, we know that for each $t>0$,
	\[
		\log \U(t,\cdot ;g)-\log \U(t,0;g)\stackrel{\text{law}}{=} B(\cdot),
	\]
	so
	\[
		\begin{aligned}
			Y_L(t) & \stackrel{\text{law}}{=}\E_W\log \int_0^L e^{B(x)+W(x)}dx                                          \\
			       & \stackrel{\text{law}}{=}\log L+\E_{\mathcal{W}} \log\int_0^1 e^{\sqrt{L}(\B(x)+\mathcal{W}(x))}dx,
		\end{aligned}
	\]
	where $\B,\mathcal{W}$ are Brownian bridges on $[0,1]$ and we used the
	scaling property of the Brownian bridges in the second equality in law. By the Laplace principle, we have
	\[
		\frac{1}{\sqrt{L}} \log\int_0^1 e^{\sqrt{L}(\B(x)+\mathcal{W}(x))}dx\to \max_{x\in[0,1]}\{\B(x)+\mathcal{W}(x)\}
	\]
	almost surely, as $L\to\infty$. Since
	\[
		\min_{x\in[0,1]}\{\B(x)+\mathcal{W}(x)\}\leq  \frac{1}{\sqrt{L}} \log\int_0^1 e^{\sqrt{L}(\B(x)+\mathcal{W}(x))}dx \leq \max_{x\in[0,1]}\{\B(x)+\mathcal{W}(x)\},
	\]
	the laws of the random variable $\frac{1}{\sqrt{L}} \log\int_0^1 e^{\sqrt{L}(\B(x)+\mathcal{W}(x))}dx$ and its square are tight as $L\to\infty$, so we can pass to the limit to complete the proof.
\end{proof}

To summarize, we have decomposed the centered height function into two parts
\begin{equation}
	\log \U(t,0;g)-\E_g\EE \log \U(t,0;g)=I_L(t)+J_L(t),\label{eq:decompose}
\end{equation}
both parts of mean zero. (See \eqref{e.defI} and \eqref{e.defJ}.) For any $t,L>0$, by Corollary~\ref{c.co1}, we have
\begin{equation}\label{eq:VarIbd}
	\Var[I_L(t)]=t\sigma_L^2,
\end{equation}
with $\sigma_L^2$ given in \eqref{e.defsigma}.  By Lemma~\ref{l.varYL}, we have
\begin{equation}\label{eq:VarJbd}
	\Var[J_L(t)] \leq CL
\end{equation}
for some universal constant $C>0$.
 In Section~\ref{sec.sigmaasympt} below, we will show (Theorem~\ref{prop:sigmaprop}) that
\begin{equation}\label{e.bdsigma}
	C^{-1}L^{-1/2}\leq \sigma^2_L \leq CL^{-1/2}, \quad\quad L\geq1,
\end{equation}
with some universal constant $C>0$. That will be the last ingredient in the following proof of Theorem~\ref{t.mainth}.

\begin{proof}[Proof of Theorem~\ref{t.mainth}]Recall that we choose $L=\lambda t^{\alpha}$, where
        $\lambda>0$ is a given constant. Using \eqref{eq:decompose},  for any $x\in \bT_L$, we write 
	\[
	\log \U(t,x;g )-\E_g\EE \log \U(t,x;g)=I_L(t)+J_L(t)+K_L(t,x),
	\]
	with 
	\[
	K_L(t,x)=\log \U(t,x;g)-\log \U(t,0;g).\]
	Here we have used the fact that $\E_g\EE \log \U(t,x;g)=\E_g\EE \log \U(t,0;g)$, which comes from the fact that, for each $t>0$,  $K_L(t,\cdot)=\log \U(t,\cdot\,;g)-\log \U(t,0;g)$ has the same law as $B(\cdot)$ and $\E B(x)=0$. This also implies that 
$$
\Var K_L(t,x)=\Var B(x)=x(L-x)/L\leq L/4.
$$
The rest of the proof is based on applying \eqref{eq:VarIbd}, \eqref{eq:VarJbd}, \eqref{e.bdsigma} and the triangle inequality
	\[
	\begin{aligned}
    \sqrt{\Var I_L(t)}-\sqrt{\Var [J_L(t)+K_L(t,x)]}&\leq \sqrt{\Var \log \U(t,x;g)}\\
     &\leq  \sqrt{\Var I_L(t)}+\sqrt{\Var [J_L(t)+K_L(t,x)]}.
		\end{aligned}
	\]
	If $\alpha\in [0,2/3)$, then we have $\Var [J_L(t)+K_L(t,x)]\ll \Var I_L(t)\asymp t^{1-\frac{\alpha}{2}}$, which implies that \[\Var \log\U(t,x;g)  \asymp t^{1-\frac{\alpha}{2}}.\] If $\alpha=2/3$, there is a constant $C>0$ so that
	\[
C^{-1}\lambda^{-1/2} t^{2/3}\leq  \Var I_L(t) \leq
    C\lambda^{-1/2}t^{2/3} \] and
 \[\Var
                                                          [J_L(t)+K_L(t,x)]
                                                          \le C\lambda t^{2/3}.
\]
	We clearly have the upper bound $\Var \log \U(t,x;g)\les t^{2/3}$. By choosing $\lambda$ small so that
	\[
		 C^{-1}\lambda^{-1/2}>  C\lambda,
	\]
	we also obtain a lower bound. The proof is complete.	
\end{proof}

\begin{remark}\label{r.sub}
	We note that \eqref{e.conjecture1}  covers the super-relaxation regime and only a part of the relaxation regime. The natural prediction in the relaxation and sub-relaxation regime is $\Var \,\log \U(t,0)\asymp t^{2/3}$. To further refine our result and to possibly cover the sub-relaxation regime, one needs to study the term $J_L(t)$  on the r.h.s.\ of \eqref{eq:decompose}. For fixed $t>0$, sending $L\to\infty$ in \eqref{eq:decompose}, the l.h.s.\ should converge to the corresponding centered solution of the KPZ on the whole line, starting from a two-sided Brownian motion. For the first term on the r.h.s., we have $I_L(t)\to0$ as $L\to\infty$, which comes from the fact that $\Var[I_L(t)]=t\sigma_L^2\to0$
	by \eqref{e.bdsigma}. This implies that the second term on the r.h.s., $J_L(t)=Y_L(0)-Y_L(t)$, must converge as $L\to\infty$. Therefore, to study the fluctuations   in the sub-relaxation regime, it is necessary to understand the cancellation in $Y_L(0)-Y_L(t)$, because the process $\{Y_L(t)\}_{t\geq0}$ is stationary, and as shown in Lemma~\ref{l.varYL}, its marginal distribution does not converge as $L\to\infty$.
\end{remark}


\section{Asymptotics of \texorpdfstring{$\sigma_L^2$}{\_L\^2}}\label{sec.sigmaasympt}
Our goal in this section is to estimate the quantity $\sigma_L^2$ appearing in \eqref{e.defsigma}. For $j=1,2$, we define the Brownian bridge $U_j = B_j+B_3$, so we can rewrite \eqref{e.defsigma} as \begin{equation}
	\sigma^2_L = \Er\left[\frac{\int_0^L e^{(U_1+U_2)(x)} dx}{\int_0^L e^{U_1(x)} dx\int_0^L e^{U_2(x)} dx}\right].
\end{equation}

\begin{theorem}\label{prop:sigmaprop}
	There exists a constant $C>0$ so that, for all $L\ge 1$,
	\[
		C^{-1}L^{-1/2}\le \sigma^2_L\le CL^{-1/2}.
	\]
\end{theorem}
The upper bound will be proved as Proposition~\ref{prop:ub}, and the lower bound will be proved as Proposition~\ref{prop:lb}. Throughout this section, we will let $C$ represent a positive, finite constant that is always independent of $L$ but may change from line to line in a computation. In the next two subsections, we introduce some preliminary notions.

\subsection{Translation-invariance}\label{subsec:transinv} We first make a simple reduction using the circular translation-invariance of the modified Brownian bridge with mean zero rather than endpoints equal to zero.
\begin{lemma}\label{l.BBstationarity}
	We have
	\begin{equation}
		\sigma^2_L=L \Er\left[\frac{1}{\int_0^L e^{ U_1(x)} dx\int_0^L e^{ U_2(x)} dx}\right].\label{e.use.trans.inv}
	\end{equation}
\end{lemma}
\begin{proof}
	Let
	\[
		D_j = \frac1L\int_0^L U_j(x)dx.
	\]
	Then $(\widetilde{U}_1,\widetilde{U}_2)\coloneqq
        (U_1-D_1,U_2-D_2)$ is invariant under circular shifts of the
        interval $[0,L]$, i.e.\ under translations considered on  $\R/(L\Z)$. Therefore, we have
	\begin{align*}
		\sigma^2_L = \Er\left[\frac{\int_0^L e^{(\widetilde U_1+\widetilde U_2)(x)} dx}{\int_0^L e^{\widetilde U_1(x)} dx\int_0^L e^{\widetilde U_2(x)} dx}\right] & = L \Er\left[\frac{e^{(\widetilde U_1+\widetilde U_2)(0)}}{\int_0^L e^{\widetilde U_1(x)} dx\int_0^L e^{\widetilde U_2(x)} dx}\right]
		\notag                                                                                                                                                                                                                                                                                             \\&= L \Er\left[\frac{1}{\int_0^L e^{ U_1(x)} dx\int_0^L e^{ U_2(x)} dx}\right]\label{e.use.trans.inv}
	\end{align*}and hence \eqref{e.use.trans.inv}.
\end{proof}

\begin{remark}The decay rate $\sigma_L^2 \asymp L^{-1/2}$ results from the   correlation between $U_1$ and $U_2$. For two extreme cases  (i) $U_1=U_2$ and (ii) $U_1$ being independent of $U_2$, one can carry out explicit calculations. In case (i), by the above lemma, we can write
\[
	\sigma_L^2=L \Er\left[\frac{1}{(\int_0^L e^{ U_1(x)} dx)^2}\right].
\]
By the calculation in \cite[Proposition 4.1]{GK211}, we have $\sigma_L^2 \asymp 1$. In case (ii), similarly we have
\[
	\sigma_L^2=L \bigg(\Er\left[\frac{1}{\int_0^L e^{ U_1(x)} dx}\right]\bigg)^2.
\]
By another application of the translation-invariance as in the proof of Lemma~\ref{l.BBstationarity} (or see e.g.\  \cite[Proposition 5.9]{MY05}), we actually have in this case that $\sigma_L^2=L^{-1}$. Therefore, the optimal decay rate in Theorem~\ref{prop:sigmaprop} comes from a delicate analysis which hinges on the correlation between $U_1,U_2$. Our analysis also applies to the case when $U_1$ and $U_2$ are Brownian bridges with a general positive scalar correlation $r$ (i.e.\ $\E [U_1(x)U_2(x)] = r\sqrt{\E U_1(x)^2\E U_2(x)^2}$), in which case the corresponding decay rate would be
 \begin{equation}
  L \E\left[\frac{1}{\int_0^L e^{ U_1(x)} dx\int_0^L e^{ U_2(x)} dx}\right]\asymp L^{1-\frac{\pi}{\pi-\arccos(r)}}.
 \end{equation}
 \end{remark}

\subsection{Change of variables}\label{subsec:chgvar}
We now make a further simplification by performing an affine change of coordinates so that we work with two independent standard Brownian bridges rather than two correlated Brownian bridges. From now on, we identify $\R^2$ and $\CC$. A straightforward covariance calculation shows that if we define, for $k=1,2$,
\[
	v_k = -\frac{  \sqrt2}2(\sqrt{3},(-1)^k) = - \sqrt2e^{(-1)^k i\pi/6},
\]
then there is a standard two-dimensional Brownian bridge $V=(V_1,V_2)$ so that
\[
	v_k\cdot V = U_k,\quad\quad k=1,2.
\]

We  define the probability measure $\Prm_{x_1,u_1}^{x_2,u_2}$ to be the law of a standard two-dimensional Brownian bridge $V$ on $[x_1,x_2]$ with $V(x_\ell) = u_\ell$, $\ell=1,2$. Thus $\Prm = \Prm_{0,0}^{L,0}$. We let $\Er_{x_1,u_1}^{x_2,u_2}$ be the corresponding expectation.

Given our change of variables, we can rewrite \eqref{e.use.trans.inv} as
\begin{equation}
	\sigma_L^2 = L\Er\left[\prod_{k=1}^2 \left(\int_0^L e^{v_k\cdot V(x)}dx\right)^{-1}\right].\label{e.prod.using.V}
\end{equation}
Define, for $v\in\R^2$,
\begin{equation}\label{eq:omegadef}
	\omega(v)=\max_{k\in\{1,2\}}[v_{k}\cdot v].
\end{equation}
Thus for any $x\in[0,L]$, we have $\omega(V(x))=\max\limits_{k=1,2} U_k(x)$. This means that upper bounds on $\omega(V(x))$, such as in the barriers we construct below, translate directly to upper bounds on the correlated Brownian bridges $U_1,U_2$.

\subsection{Brownian motion in a wedge}\label{subsec:bminwedge}
Roughly speaking, the main contribution to the expectation in \eqref{e.prod.using.V} comes from those paths of $V$ such that
\[
	\omega(V(x))=\max_{k\in\{1,2\}}U_k(x)
\] is bounded from above, uniformly for all $x\in[0,L]$. This is because, if there is even an order-$O(1)$ interval of time such that $\omega(V(x))$ is large, this will cause the corresponding integral inside the expectation in \eqref{e.prod.using.V} to be exponentially large, and thus the contribution to the expectation to be exponentially small. To understand the probability that $\sup\limits_{x\in[0,L]}\omega(V(x))\leq a$ for a given barrier $a$, we will need some facts about two independent Brownian motions killed upon exiting a wedge. The angle of the wedge will come from the particular correlation between $U_1$ and $U_2$.

For $a\in\R$, we define the translated wedge
\begin{equation}
	N_a\coloneqq \omega^{-1}((-\infty,a]) = \{v\in\R^2\mid v_k\cdot v\le a\text{ for }k=1,2\}.
\end{equation}
We use the convention that $N_\infty = \R^2$, and define
\[N\coloneqq N_0 = \{re^{i\theta}\mid r\ge 0\text{ and }\theta\in [-\pi/3,\pi/3]\}.\]

Now fix the vector
\begin{equation}\label{eq:hdef}h = (-\sqrt6/3,0).\end{equation} Since $v_k\cdot h = 1$ for $k=1,2$, we have for all $a\in\R$ and $v\in \R^2$ that
\begin{equation}\omega(v+ah)=\omega(v)+a.\label{eq:hshift}\end{equation} Thus, we have
\begin{equation}
	N_a = N+ah.\label{eq:translateN}
\end{equation}See Figure~\ref{fig:wedgegeom} for an illustration of the geometry.

\begin{figure}
	\begin{tikzpicture}
		\clip (0,0) circle[radius=2.5];
		\fill[opacity=0.1] (3,{-2*sqrt(3)})--(2,{-2*sqrt(3)})--(0,0)--(2,{2*sqrt(3)})--(3,{2*sqrt(3)});
		\draw (2,{-2*sqrt(3)})--(0,0)--(2,{2*sqrt(3)});
		\draw ({2-sqrt(6)/2},{-2*sqrt(3)})--({-sqrt(6)/2},0)--({2-sqrt(6)/2},{2*sqrt(3)});
		\fill[opacity=0.05] (3,{-2*sqrt(3)})--({2-sqrt(6)/2},{-2*sqrt(3)})--({-sqrt(6)/2},0)--({2-sqrt(6)/2},{2*sqrt(3)})--(3,{2*sqrt(3)});
		\node at (1.1,-1.1) {$N_a$};
		\node at (-0.1,-1) {$N_{a+1}$};
		\draw[->,>=latex,color=blue] ({0.5-sqrt(6)/2},{sqrt(3)/2})--({0.5-sqrt(6)/2-0.5*sqrt(6)/2},{sqrt(3)/2+0.5*sqrt(2)/2}) node[anchor=east] {$v_1$};
		\draw[->,>=latex,color=blue] ({0.5-sqrt(6)/2},{-sqrt(3)/2})--({0.5-sqrt(6)/2-0.5*sqrt(6)/2},{-sqrt(3)/2-0.5*sqrt(2)/2}) node[anchor=east] {$v_2$};
		\draw[->,>=latex,color=blue] (0,0)-- node[anchor=south] {$h$} ++({-sqrt(6)/2},0);
	\end{tikzpicture}
	\caption{The vectors $v_1$ and $v_2$ are perpendicular to the edges of the wedges $N_a$, and the vector $h$ is the offset between $N_a$ and $N_{a+1}$.}\label{fig:wedgegeom}
\end{figure}
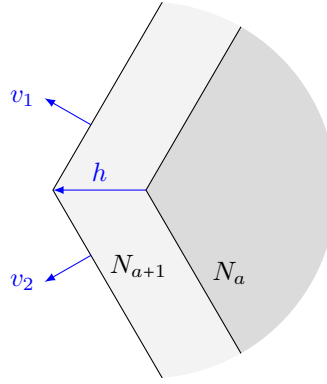

With the above definitions, we see that for any $a\in \R$ and any
interval $I\subseteq [0,L]$, the event that
	\[
		\max_{x\in I }[ U_1(x)\vee U_2(x)]\leq a
	\]
	translates to the event that $V(x)\in N_a$ for all $x\in I$, i.e.\ that the independent Brownian bridges stay in the wedge on the interval $I$.
	Thus, for $a>0$ and (possibly random) $0\le x_1<x_2\le L$, we define the event
\begin{equation}\label{eq:Edef}
	E_{a,x_1,x_2} = \{V(x)\in N_a\text{ for all }x\in [x_1,x_2]\}.
\end{equation}
By an elementary property of the Brownian bridge, we see that $\Prm[E_{0,1,L}]=0$, and recalling the convention that $N_\infty=\R^2$, we see that $\Prm[E_{\infty,x_1,x_2}]=1$ for any $x_1,x_2$.

We  recall the standard two-dimensional heat kernel
 \begin{equation}
	p_x(u_1,u_2)=(2\pi x)^{-1}e^{-|u_1-u_2|^2/(2x)}, \quad\quad x>0,u_1,u_2\in\R^2.
	\label{eq:stdhk}
\end{equation}
We now introduce the heat kernel of the two-dimensional standard Brownian motion killed on the boundary of the wedge $N$.
Let  $p_x^N$ be the transition kernel of standard two-dimensional Brownian motion in $N$ killed on $\partial N$, so, if $Y$ is a Brownian motion with $Y(0)=u\in N$, then
\[
	p_x^N(u,v) = \Prm[Y(x)\in dv\text{ and } Y([0,x])\subset N].
\]
In other words, for any subset $A\subset N$, we have
	\[
		\Prm[Y(x)\in A\text{ and } Y([0,x])\subset N]=\int_A p_x^N(u,v)dv.
	\]
The following explicit formula for $p_x^N$, with the arguments in polar coordinates, is known (see  \cite[p.~379]{CJ88} or \cite[Lemma~1]{BS97}): for $r_1,r_2>0$ and $\theta_1,\theta_2\in [-\pi/3,\pi/3]$,
\begin{equation}\label{eq:ptNformula}
	p_x^N(r_1e^{i\theta_1},r_2e^{i\theta_2}) =\frac1x\sum_{j=1}^{\infty}I_{\frac{3}{2}j}\left(\frac{r_1 r_2}{x}\right)\prod_{k=1}^2 \left[e^{-r_k^2/(2x)}\sin\left(\frac{3}{2}j(\theta_k+\pi/3)\right)\right].\end{equation}
Here, $I_{\alpha}(\nu)$ is the modified Bessel function of the first
kind. We review the properties of $I_\alpha(\nu)$ that we will need in
Appendix~\ref{appdx:bessel}. At this point, it is worth emphasizing
that the angle of the wedge $N$, which is $2\pi/3$, enters the formula
\eqref{eq:ptNformula} through the degrees of the Bessel functions and
the arguments in the sine functions. For Brownian bridges with different correlations, corresponding to  wedges with a different angle, the formula \eqref{eq:ptNformula} would be different.
We further define
\begin{equation}
	p_x^{N_a}(u_1,u_2)=p_x^N(u_1-ah,u_2-ah),\label{eq:translateit}
\end{equation}
which by \eqref{eq:translateN} is the transition kernel for Brownian motion on $N_a$ killed on $\partial N_a$.

\subsubsection{Estimates on $p_x^N$} We now prove some estimates on the heat kernel $p_x^N$. We will use the estimates on modified Bessel functions of the first kind proved in Appendix~\ref{appdx:bessel}. The bounds here are suited to understanding $p_x^N(u_1,u_2)$ when $x$ is not too small relative to $u_1$ and $u_2$, which is all we need in our analysis. The small-$x$ regime requires different analysis; see for example \cite{CBM15}.
\begin{lemma}
	There exists a constant $C>0$ so that for all $x,r_1,r_2>0$ and
	$\theta_1,\theta_2\in[-\pi/3,\pi/3]$ we have
	\begin{align}
		p_x^N(r_1e^{i\theta_1},r_2e^{i\theta_2}) & \le \frac{C(r_1r_2)^{3/2}  }{x^{5/2}}e^{(Cr_1r_2-r_1^2-r_2^2)/(2x)} \prod_{k=1}^2 (\pi/3-|\theta_k|).\label{eq:ptxybd}
	\end{align}
\end{lemma}

\begin{proof}
	Using the   identity $|\sin\theta|\le|\theta|$, we have for
	all $\theta\in[-\pi/3,\pi/3]$ that
	\begin{align*}
		\left|\sin\left(\frac{3}{2}j(\theta+\pi/3)\right)\right| & \le\frac{3}{2}j\left(|\pi/3+\theta|\wedge|\pi/3-\theta|\right)=\frac32 j\left(\pi/3-|\theta|\right).
	\end{align*}
	Using this  and Lemma~\ref{lem:besselsumbd} in \eqref{eq:ptNformula}, we obtain
	\begin{align*}
		p_x^N(r_1e^{i\theta_1},r_2e^{i\theta_2}) & \le\frac{9}{4x} \left(\prod_{k=1}^2 \left[e^{-r_k^2/(2x)}(\pi/3-|\theta_k|)\right]\right) \sum_{j=1}^{\infty}j^2I_{\frac32 j}\left(\frac{r_1 r_2}{x}\right) \\
		                                         & \le \frac{C(r_1r_2)^{3/2}}{x^{5/2}}e^{-(r_1^2+r_2^2)/(2x)+Cr_1r_2/x} \prod_{k=1}^2 (\pi/3-|\theta_k|),                                                      \end{align*}
	which is \eqref{eq:ptxybd}.
\end{proof}

\begin{lemma}
	There exists a constant $C>0$ so that, for all $x,a>0$, we have
	\begin{equation}
		p_x^{N_a}(0,0)= p_{x}^{N}(-ah,-ah) \ge C^{-1}a^3x^{-5/2}e^{-Ca^2/x}.\label{eq:ptNlowerbound}
	\end{equation}
\end{lemma}

\begin{proof}
	Recall that $h=(-\sqrt6/3,0)$.   We use \eqref{eq:ptNformula}
	to write
	\[
		p_{x}^{N}(-ah,-ah)=\frac{1}{xe^{\frac{2a^2}{3x}}}\sum_{j=1}^{\infty}I_{\frac32 j}\big(\frac{2a^2}{3x}\big)\sin^2\big(\frac{\pi j}{2}\big).\]
	Since all of the terms of the above series are positive, it
        suffices to consider the first one and apply Lemma~\ref{lem:besselbd} to complete the proof.
\end{proof}

A consequence of the above two lemmas is
	\begin{corollary}\label{c.bdEa}
		There exists a constant $C>0$ so that, for all $L\ge 1$, we have
		\begin{equation}\label{eq:PEbds}
			\frac{a^3}{CL^{3/2}}e^{-Ca^2/L}\le\Prm[E_{a,0,L}]\le \frac{Ca^3}{L^{3/2}}.
		\end{equation}
	\end{corollary}

\begin{proof}
		First, we have  
		\[
			{\Prm[E_{a,0,L}]=\frac{p_L^{N_a}(0,0)}{p_L(0,0)}}=2\pi L p_L^{N_a}(0,0)=2\pi L p_L^{N}(-ah,-ah).
		\]
		where the last ``='' is by \eqref{eq:translateit}. Further applying \eqref{eq:ptxybd} and \eqref{eq:ptNlowerbound} yields
		\[
			\frac{a^3}{CL^{3/2}}e^{-Ca^2/L}\leq \Prm[E_{a,0,L}]\leq C\frac{a^3}{L^{3/2}}e^{Ca^2/L}
		\]
		with some constant $C>0$. The seemingly improved upper bound in \eqref{eq:PEbds} follows when we note that a probability is always at most $1$.
	\end{proof}

\subsection{Upper bound}\label{subsec:ub}
In this subsection we prove the upper bound claimed in the statement of Theorem~\ref{prop:sigmaprop}, namely the following proposition.

\begin{proposition}\label{prop:ub}
	There exists a constant $C>0$ so that, for all $L\geq1$, we have $\sigma_L^2\le CL^{-1/2}$.
\end{proposition}

By increasing the constant $C$ if necessary, we note that it suffices to prove this proposition for $L$ greater than some finite constant, so in the proof we will assume somewhat larger (constant) lower bounds on $L$ when it is convenient.

In order to prove Proposition~\ref{prop:ub}, we will partition the probability space according to the maximum value of $\omega(V(x))=\max_{k\in\{1,2\}}[U_k(x)]$, as $x$ varies in $[0,L]$. By Corollary~\ref{c.bdEa}, the probability that $\omega(V(x))\le a$ for all $x\in [0,L]$ is of order $L^{-3/2}a^3$ if $L$ is of the order $O( a^2)$. On the other hand, as mentioned in the previous subsection, if $\omega(V(x))$ is of the order $a$ at some point $x$, then the contribution to the expectation in \eqref{e.prod.using.V} should be exponentially small in $a$. Thus we should be able to integrate in $a$ to achieve a bound of the order $L^{-3/2}$, which will yield $\sigma_L^2\le L^{-1/2}$ when multiplied by $L$ as in \eqref{e.prod.using.V}.

\subsubsection{Scales and stopping times}
Our first step in carrying out our program for the upper bound is to define a sequence of scales for the maximum of $\omega(V(x))$, as $x$ varies in $[0,L]$.
Let 
\begin{equation}
\label{xi0}
\xi_0=100,  \quad \delta=1/100\quad\mbox{ and $M =
  \lfloor{\delta\log_{\xi_0} L}\rfloor + 1$}.
\end{equation}
 For $m\in\{0,1,\ldots,M\}$, define
\begin{equation}\label{eq:qmdef}
	q_m = \begin{cases}\xi^m_0-1,&0\le m <M;\\ \infty,&m= M.\end{cases}.
\end{equation}
In particular, this means that (whenever $L\ge 10$)
\begin{equation}
	1\le q_m \le L^\delta\le L^{1/2}\le L/2\text{ for all }m\in\{1,\ldots,M-1\}. \label{eq:bulkqbounds}
\end{equation}
For $m\in\{1,2,\ldots,M\}$, define (recalling the definition \eqref{eq:Edef})
\begin{equation}\label{eq:Amdef}
	A_m = E_{q_m,0,L}\setminus E_{q_{m-1},0,L},
\end{equation}
which is the event that the maximum of $\omega(V(x))$, $x\in[0,L]$, is between $q_{m-1}$ and $q_m$.
Note that the events $A_m$ are pairwise disjoint and
\[
	\Prm\left[\bigcup_{m=1}^M A_m\right] = 1-\Prm[E_{q_0,0,L}] = 1-\Prm[E_{0,0,L}] = 1.
\]
The last ``='' comes from the fact that  the Brownian bridge $V$ starts at $0$ and thus will exit the cone $N$ immediately with probability $1$.
Therefore, we have by \eqref{e.prod.using.V} that
\begin{equation}
	\sigma_L^2 = L\sum_{m=1}^M G_m,\label{eq:breakupsigmaL2}
\end{equation}
where
\begin{equation}
	G_m\coloneqq \Er\left[\prod_{k=1}^2 \left(\int_0^L e^{v_k\cdot V(x)}dx\right)^{-1};A_m\right].\label{eq:Gmdef}
\end{equation}

Now we define a sequence of stopping times corresponding to our scales
$q_m$. We will need to consider stopping times with respect to both
the forward and the backward canonical filtrations of the Brownian
bridge $V$. Let $\{\mathcal{F}_x\}_{x\in [0,L]}$ be the canonical
forward filtration, so $\mathcal{F}_x$ is the $\sigma$-algebra
generated by $V|_{[0,x]}$, and let
$\{\widetilde{\mathcal{F}}_x\}_{x\in [0,L]}$ be the canonical backward
filtration, so $\widetilde{\mathcal{F}}_x$ is the $\sigma$-algebra
generated by $V|_{[x,L]}$. We note that $V(x)/(1-x/L)$, $0\le x<L$, is a continuous
$\{\mathcal{F}_x\}$-martingale with the  covariation matrix of the
form $x(L-x)^{-1}I_2$. Here $I_2$ is the $2\times 2$ identity
matrix. Hence it is a time-change of a two-dimensional standard Brownian
motion.  Likewise, $V(x)/(x/L)$, $0< x\le L$,  is a continuous backward
$\{\widetilde{\mathcal{F}}_x\}$-martingale,   with the
  covariation matrix   $(L-x)x^{-1}I_2$. We define our stopping times for these martingales rather than  $V$ because it will aid the computation of certain harmonic measures; see Lemma~\ref{lem:computeharmonicmeasure} below.

\begin{definition}For $1\le m\le M$, let $\tau_m$ (resp
  $\widetilde\tau_m$) be the first (resp. last) time $x\in [0,L]$ such
  that $\omega(V(x))/(1-x/L) = q_{m-1}$
  (resp. $\omega(V(x))/(x/L)=q_{m-1}$). It is clear that, with
  probability $1$, such stopping times  exist.
	Let $k_m$ (resp. $\widetilde{k}_m$) be such that $\omega(V(\tau_m)) = v_{k_m}\cdot V(\tau_m)$ (resp. $\omega(V(\widetilde\tau_m))=v_{\widetilde k_m} \cdot V(\widetilde\tau_m)$).
\end{definition}
From the definition, we see that $\tau_m$ is a $\{\mathcal{F}_x\}$-stopping time and $\widetilde\tau_m$ is a $\{\widetilde{\mathcal{F}}_x\}$-stopping time.
With the scales and stopping times defined, we are ready to set up our framework for bounding $G_m$.
\begin{lemma}\label{lem:GCSsym}
	For all $L\ge 6$ and all $1\le m\le M$, we have
	\begin{equation}\label{eq:Gm2bd}
		G_m^2 \le 4e^{- q_{m-1}}Q_{q_m}(L,0)\Er[Q_{q_m}(L-\tau_m,V(\tau_m));\tau_m\le L/2],
\end{equation}
	where
	\begin{equation}\label{eq:Qmdef}
		Q_q(J,u)\coloneqq \Er_{0,u}^{J,0}\left[\int_0^1 e^{2\sqrt2 |V(x)-u|}dx;E_{q,0,J}\right].
	\end{equation}
\end{lemma}

Before we prove Lemma~\ref{lem:GCSsym}, let us briefly explain its interpretation. By definition, $G_m$ represents the contribution to $\sigma_L^2$ from the paths $V$ such that the maximum of $\omega(V(x))$ is between  $q_{m-1}$ and $q_m$. The term $e^{- q_{m-1}}$ thus bounds the size of the contribution we expect on this event.
The terms $Q_{q_m}$ encapsulate two different effects: first, the error in upper-bounding a unit-length integral by the value of the integrand at the starting point (which should be of order $1$), and second, the probability of $\omega(V(x))$ staying below $q_m$ (which is roughly of order $L^{-3/2}$ times a polynomial in $q_m$ -- see Corollary~\ref{c.bdEa}). The restriction $\tau_m\le L/2$ in \eqref{eq:Gm2bd} is possible by a forward/backward symmetrization argument and will be important in the sequel.

\begin{proof}[Proof of Lemma~\ref{lem:GCSsym}]
	For any $x_1,x_2\in\R$, define
	\[Z_k(x_1,x_2) = \left(\int_{x_1\vee 0}^{x_2\wedge L} e^{v_k\cdot V(x)}dx\right)^{-1}.\]
	We first note that
	\[
		Z_1(0,L)\vee Z_2(0,L) \le Z_1(0,1) \vee Z_2(0,1)
	\]
	and
	\[
		Z_1(0,L)\wedge Z_2(0,L)\le Z_{k_m}(\tau_m,\tau_m+1)\wedge Z_{\widetilde k_m}(\widetilde\tau_m-1,\widetilde\tau_m)
	\]
 for all $1\le m\le M$.
	Thus we can estimate the random variable inside the expectation in \eqref{eq:Gmdef} by
	\begin{align*}
		\prod_{k=1}^2 & \left(\int_0^L e^{v_k\cdot V(x)}dx\right)^{-1}                                                                                             = (Z_1(0,L)\vee Z_2(0,L))(Z_1(0,L)\wedge Z_2(0,L)) \\
		              & \le \left(Z_1(0,1) \vee Z_2(0,1)\right)\left(Z_{k_m}(\tau_m,\tau_m+1)\wedge Z_{\widetilde k_m}(\widetilde\tau_m-1,\widetilde\tau_m)\right).
	\end{align*}
	Using this in \eqref{eq:Gmdef} and then  applying  the Cauchy--Schwarz inequality, we obtain
	\begin{equation}
		\begin{aligned}
			G_m^2 \le & \Er[Z_1(0,1)^2+Z_2(0,1)^2;A_m]                                                                                                      \\
			          & \times \Er\left[\left(Z_{k_m}(\tau_m,\tau_m+1)^2\wedge Z_{\widetilde k_m}(\widetilde\tau_m-1,\widetilde\tau_m)^2\right);A_m\right].
			\label{eq:CauchySchwarz}
		\end{aligned}
	\end{equation}
	Applying Jensen's inequality (recalling that $|v_k|=\sqrt{2}$), we bound the first factor on the r.h.s.\ of \eqref{eq:CauchySchwarz} by
	\begin{equation}\label{e.253}
		\Er[Z_1(0,1)^2+Z_2(0,1)^2;A_m] \leq  2Q_{q_m}(L,0).
	\end{equation}

	Now we consider the second factor on the r.h.s.\ of \eqref{eq:CauchySchwarz}. By definition, on the event $E_{q_{m-1,0,L}}^\mathrm{c}$, there exists a time $x\in [0,L]$ such that $\omega(V(x))>q_{m-1}$. This means in particular that $\omega(V(x))/(1-x/L)>q_{m-1}$ and $\omega(V(x))/(x/L)>q_{m-1}$, so if $x\in [0,L/2]$, then $\tau_m\le L/2$, while if $x\in [L/2,L]$, then $\widetilde\tau_m\ge L/2$. So in fact
	\begin{equation}\1_{E_{q_{m-1,0,L}}^\mathrm{c}}\le \1_{\tau_m\le L/2}+\1_{\widetilde\tau_m\ge L/2}.\label{eq:Eqmclttausok}\end{equation}
	Using \eqref{eq:Eqmclttausok} and recalling that $A_m=E_{q_m,0,L}\setminus E_{q_{m-1},0,L}$, we obtain
	\begin{equation}\label{e.251}
		\begin{aligned}
			 & \Er\left[\left(Z_{k_m}(\tau_m,\tau_m+1)^2\wedge Z_{\widetilde k_m}(\widetilde\tau_m-1,\widetilde\tau_m)^2\right);A_m\right]                                                                               \\
			 & \leq \Er \left[Z_{k_m}(\tau_m,\tau_m+1)^2   \1_{\tau_m\le L/2}; E_{q_m,0,L}\right]+\Er \left[Z_{\widetilde k_m}(\widetilde\tau_m-1,\widetilde\tau_m)^2   \1_{\widetilde\tau_m\ge L/2}; E_{q_m,0,L}\right] \\
			 & = 2\Er \left[Z_{k_m}(\tau_m,\tau_m+1)^2   \1_{\tau_m\le L/2}; E_{q_m,0,L}\right],
		\end{aligned}
	\end{equation}
	where the last ``='' is by the time-reversal symmetry of the Brownian bridge.
	Note that when $\tau_m\le L/2$, we have $[\tau_m,\tau_m+1]\subset [0,L]$, in which case
	\[
		Z_{k_m}(\tau_m,\tau_m+1)=\left(\int_{\tau_m}^{\tau_m+1} e^{v_k\cdot V(x)}dx\right)^{-1}.
	\]
	Therefore, we apply Jensen's inequality again, which leads to
	\begin{equation}\label{e.252}
		\begin{aligned}
			\Er\big[\big(Z_{k_m}(\tau_m, & \tau_m+1)^2\wedge Z_{\widetilde k_m}(\widetilde\tau_m-1,\widetilde\tau_m)^2\big);A_m\big] \\
			                             & \leq 2  \Er\left[Z_{k_m}'(\tau_m,\tau_m+1);E_{q_m,0,L}\cap\{\tau_m\le L/2\}\right]        \\
			                             & =2  \Er\left[Z_{k_m}'(\tau_m,\tau_m+1);E_{q_m,\tau_m,L}\cap\{\tau_m\le L/2\}\right],
		\end{aligned}
	\end{equation}
	with
	\[
		Z'_k(x_1,x_2) = \int_{x_1}^{x_2}e^{-2v_k\cdot V(x)}dx.
	\]
	For the last ``='', it is because,
	if $x\le\tau_m\le L/2$, then (i) if $\omega(V(x))\geq0$, we have \begin{equation}
		\omega(V(x))\le \frac{\omega(V(x))}{1-x/L}\le q_{m-1}\le q_m,\label{eq:omegaVtaum}\end{equation}
	 and (ii) if $\omega(V(x))<0$, we also have $\omega(V(x))<q_m$.
	This implies  \[E_{q_m,0,L}\cap\{\tau_m\le L/2\} = E_{q_m,\tau_m,L}\cap\{\tau_m\le L/2\}.\]

	Now since (still on the event $\tau_m\le L/2$) we
        have \[v_{k_m}\cdot V(\tau_m) = (1-\tau_m/L) q_{m-1} \ge
          \frac12 q_{m-1},\] we can further bound from above  the
        r.h.s.\ of \eqref{e.252} by (recalling that $|v_k|=\sqrt{2}$
        and using the strong Markov property of $V$)
	\begin{align*}
		2\Er & \left[e^{-2v_{k_m}\cdot V(\tau_m)}\int_{\tau_m}^{\tau_m+1}e^{2\sqrt{2}|V(x)-V(\tau_m)|}dx;E_{q_m,\tau_m,L}\cap\{\tau_m\le L/2\}\right] \\
		     & \le 2e^{-q_{m-1}}\Er[Q_{q_m}(L-\tau_m,V(\tau_m));\tau_m\le L/2].
	\end{align*}
	Using this in \eqref{e.252} and combining with \eqref{e.253}, we complete the proof.
\end{proof}

\subsubsection{The quantity $Q_q$}
In order to use \eqref{eq:Gm2bd} to prove Proposition~\ref{prop:ub}, we need to estimate $Q_{q}(J,u)$. This is done in the next lemma.
\begin{lemma}\label{lem:Qbd}
	There is a constant $C$ so that for any $J\ge 2$, any $q\in [1,J]$, and any $u\in\R^2$ we have
	\begin{equation}
		Q_q(J,u)\le CJ^{-3/2}q^{3/2}e^{CqJ^{-1}(q+|u|)}\big(q^{1/2}+|u|^{1/2}\Big)\Big(q+|\omega(u)|\Big).\label{eq:Qbd}
	\end{equation}
	Moreover, for any $q\in (0,\infty]$, $J\ge 2$, and $u\in\R^2$, we have
	\begin{equation}
		Q_q(J,u)\le Ce^{C|u|/J}.\label{eq:Qbddbyexp}
	\end{equation}

\end{lemma}
Thanks to our symmetrization procedure in the proof of Lemma~\ref{lem:GCSsym}, $Q_q(J,u)$ is only applied in \eqref{eq:Gm2bd} with $J\in [L/2,L]$, so $J$ is of order $O(L)$. The dependence on $u$ is somewhat more problematic, as in \eqref{eq:Gm2bd} we must take $u=V(\tau_m)$. While we know that $V(\tau_m)/(1-\tau_m/L)\in \partial N_{q_{m-1}}$, the boundary $\partial N_{q_{m-1}}$ of the wedge is non-compact, and some additional control will be needed. That will be the goal of Section~\ref{subsec:bridgeatstoppingtime} below.

\begin{proof}[Proof of Lemma~\ref{lem:Qbd}]
	Using the definition \eqref{eq:Qmdef} and the Markov property of the Brownian bridge, we have
	\begin{align}
		Q_q(J,u) & \le \Er_{0,u}^{J,0}\left[\int_0^1 e^{2\sqrt{2}|V(x)-u|}dx;E_{q,1,J}\right]\notag                                                                        \\
		         & =\Er_{0,u}^{J,0}\left[\int_0^1 e^{2\sqrt{2}|V(x)-u|}dx\times\Prm_{1,V(1)}^{J,0}[E_{q,1,J}]\right]\notag                                                 \\
		         & =\int_{N_{q}} p_{\frac{J-1}{J}}(w-(1-1/J)u) \Er_{0,u}^{1,w}\left[\int_0^1 e^{2\sqrt{2}|V(x)-u|}dx\right]\Prm_{1,w}^{J,0}[E_{q,1,J}]dw.\label{eq:Qsplit}
	\end{align}
	The second term in the product on the last line of \eqref{eq:Qsplit} is  bounded as
	\begin{equation}
		\Er_{0,u}^{1,w}\left[\int_0^1 e^{2\sqrt{2}|V(x)-u|}dx\right]= \int_0^1 \Er_{0,u}^{1,w}e^{2\sqrt{2}|V(x)-u|}dx\le Ce^{C|w-u|}\label{eq:boundintexp}
	\end{equation}
	for a universal constant $C$. Using \eqref{eq:boundintexp} and the trivial bound $\Prm_{1,w}^{J,0}[E_{q,1,J}]\le 1$ in \eqref{eq:Qsplit} implies that
	\begin{align*}
		Q_q(J,u) & \le C \int_{\R^2}p_{\frac{J-1}{J}}(w-(1-1/J)u)e^{C|w-u|}dw\le C e^{C|u|/J}\int_{\R^2}p_{\frac{J-1}{J}}(w)e^{C|w|}dw
	\end{align*}
	and hence \eqref{eq:Qbddbyexp}.

	To derive \eqref{eq:Qbd}, we proceed by noting that
	the third term in the product on the farthest right-hand side of of \eqref{eq:Qsplit} is bounded by Lemma~\ref{lem:PEbd1} below as (when $J\geq2$ and $w\in N_q$)
	\begin{equation}\label{eq:applyPEbd}
		\Prm_{1,w}^{J,0}[E_{q,1,J}]\le CJ^{-3/2}|w-qh|^{1/2}q^{3/2}(q-\omega(w))e^{CJ^{-1}q|w-qh|}.
	\end{equation}
	Using \eqref{eq:boundintexp} and \eqref{eq:applyPEbd} in \eqref{eq:Qsplit}, we obtain
	\begin{align*}
		Q_q(J,u) & \le \frac{Cq^{3/2}}{J^{3/2}}\int_{N_{q}} p_{\frac{J-1}{J}}(w-(1-1/J)u)|w-qh|^{1/2} (q-\omega(w))\notag                \\&\hspace{16em} \times e^{C[|u-w|+J^{-1}q|w-qh|]}dw\notag\\
		         & \le \frac{Cq^{3/2}}{J^{3/2}}\int_{\R^2} p_{\frac{J-1}{J}}(w)|w+(1-1/J)u-qh|^{1/2} (q-\omega(w+(1-1/J)u))_+\notag \\ & \hspace{16em} \times e^{C[|u/J-w|+J^{-1}q|w+(1-1/J)u-qh|]}dw.
	\end{align*}
	Then we collect the following estimates (which use the assumption $q\in [1,J]$)
	\begin{align*}
		|w+(1-1/J)u-qh|^{1/2}     & \le C[|w|^{1/2}+q^{1/2}+|u|^{1/2}],  \\
		(q-\omega(w+(1-1/J)u))_+ & \le C[|w|+q+|\omega(u)|],\text{ and} \\
		|u/J-w|+J^{-1}q|w+(1-1/J)u-qh| & \le C[|w|+qJ^{-1}(q+|u|)].
	\end{align*}
	The first and the third inequalities are rather straightforward to obtain. For the second one, it is enough to write
	\[
		\begin{aligned}
			q-\omega(w+(1-1/J)u)&= q-\max_{k=1,2} v_k\cdot[w+(1-1/J)u]            \\
			&=                      q+\min_{k=1,2} [-v_k\cdot w-(1-1/J)v_k\cdot u] \\
			&\leq                  q-v_{k_u}\cdot w-(1-1/J)\omega(u)              \\
			&\leq                   q+C|w|+|\omega(u)|,
		\end{aligned}
	\]
	where $k_u\in \{1,2\}$ is chosen so that $v_{k_u}\cdot u=\omega(u)$. Using the above estimates, we conclude that
	\begin{align*}
		Q_q(J,u) & \le CJ^{-3/2}q^{3/2}e^{CqJ^{-1}(q+|u|)} \int_{\R^2} p_{\frac{J-1}{J}}(w)[|w|^{1/2}+q^{1/2}+|u|^{1/2}]\notag \\ & \hspace{14em}\times [|w|+q+|\omega(u)|] e^{C|w|}dw\notag                                                                                                       \\
		         & \le CJ^{-3/2}q^{3/2}e^{CqJ^{-1}(q+|u|)}[q^{1/2}+|u|^{1/2}][q+|\omega(u)|],
	\end{align*}
	where in the last inequality we again used that $q\ge 1$. This completes the proof of \eqref{eq:Qbd}.
\end{proof}

Now we must prove the lemma used above, which is a generalization of Corollary~\ref{c.bdEa}.
\begin{lemma}\label{lem:PEbd1}
	There exits a constant $C>0$ so that for every $J>0$, $q>0$, and $w\in N_q$, we have that
	\begin{equation}\label{eq:PEbd1}
		\Prm_{0,w}^{J,0}[E_{q,0,J}]\le CJ^{-3/2}|w-qh|^{1/2}q^{3/2}(q-\omega(w))e^{CJ^{-1}q|w-qh|}.
	\end{equation}
\end{lemma}
\begin{proof}
	We note that
	\begin{equation}
		\Prm_{0,w}^{J,0}[E_{q,0,J}]=\frac{p_J^{N_q}(w,0)}{p_J(w,0)}.\label{eq:quotientthing}
	\end{equation}
	We bound the numerator using \eqref{eq:translateit} and \eqref{eq:ptxybd} (recalling that $h=(-\sqrt{6}/3,0)$)
	\begin{align}
		p_J^{N_q}(w,0) & = p_J^N(w-qh,-qh)\notag                                                                                                    \\
		               & \le \frac{C|w-qh|^{3/2}q^{3/2}}{J^{5/2}}|\sin(\pi/3-|\arg(w-qh)|)|e^{(Cq|w-qh|-|w-qh|^2-q^2|h|^2)/(2J)},\label{eq:numPEbd}
	\end{align}
	 where we used the inequality $|\pi/3-\theta| \leq C \sin (\pi/3-\theta)$ for $\theta\in[0,\pi/3]$. Here the convention  is that $\arg$ takes values in $[-\pi/2,\pi/2]$.

	Assume now without loss of generality that the second coordinate of $w$ is nonnegative, i.e., $w_2 \geq0$. This implies $\arg(w-qh)\in [0,\pi/3]$ given the assumption of $w\in N_q$. Since $\arg(v_1) = 5\pi/6$ and $|v_1| = \sqrt2$, we have
	\begin{align*}v_1\cdot (w-qh) &= \sqrt2|w-qh|\cos(5\pi/6-\arg(w-qh))\\&=-\sqrt2|w-qh|\sin(\pi/3-\arg(w-qh)).\end{align*}
This means that \[
		  0\leq \sin(\pi/3-\arg(w-qh)) = \frac{v_1\cdot(qh-w)}{\sqrt2|w-qh|} = \frac{q-v_1\cdot w}{\sqrt2|w-qh|}=\frac{q-\omega(w)}{\sqrt2|w-qh|}.\]
	In the last ``='', we used the fact that $v_1\cdot w\geq v_2\cdot w$ if $w_2\geq0$. Using this in \eqref{eq:numPEbd}, we obtain
	\begin{equation}
		p_J^{N_q}(w,0)\le\frac{C|w-qh|^{1/2}q^{3/2}}{J^{5/2}}(q-\omega(w))e^{(Cq|w-qh|-|w-qh|^2-q^2|h|^2)/(2J)}.\label{eq:finalnumbd}
	\end{equation}
	Using this, along with the expression of the standard heat kernel, in \eqref{eq:quotientthing}, we obtain
	\begin{align*}
		\Prm_{0,w}^{J,0}[E_{q,0,J}] & \le\frac{C|w-qh|^{1/2}q^{3/2}}{J^{3/2}}(q-\omega(w))e^{(Cq|w-qh|+|w|^2-|w-qh|^2-q^2|h|^2)/(2J)}. \end{align*}
	For the term in the exponent, we note that
	\begin{align*}
		Cq|w-qh|+|w|^2-|w-qh|^2-q^2|h|^2 & = Cq|w-qh|+2qh\cdot w-2q^2|h|^2       \\
		                                 & = Cq|w-qh|+2qh\cdot[w-qh]\le Cq|w-qh|
	\end{align*}
	for some larger constant $C$, so we obtain \eqref{eq:PEbd1}.
\end{proof}

\subsubsection{The location of the Brownian bridge at the stopping time}\label{subsec:bridgeatstoppingtime}
In order to apply Lemma~\ref{lem:Qbd} for the expectation in
\eqref{eq:Gm2bd}, we need  some control on
$|V(\tau_m)|$. We recall  the Gaussian tail bound
\begin{equation}
	\Prm[|V(\tau_m)|\ge\xi]\le \Prm\left[\sup_{x\in [0,L]} |V(x)|\ge \xi\right]\le Ce^{-C^{-1}L^{-1}\xi^2}.\label{eq:crudeVtaumbd}
\end{equation}
We will also need an $L$-independent bound, which we obtain in the
following lemma through a conformal transformation argument.

\begin{lemma}\label{lem:computeharmonicmeasure}
	Let $q>0$ and let $\tau$ be the first time $x$ such that
        $V(x)/(1-x/L)\in \partial N_q$. Recall that $h$ has been
        defined in \eqref{eq:hdef}.
	Then,
	\begin{equation}
		\Prm[|V(\tau)|\ge q|h|+\xi]\le \frac2\pi
                \arctan\left(\frac{q^{3/2}|h|^{3/2}}{\xi^{3/2}}\right)\quad
                \text{for all $\xi>0$.}\label{eq:Vtaubd}
	\end{equation}
\end{lemma}
\begin{proof}
	Define \[Y(y) = \frac{L+y}{L}V\left(\frac{Ly}{L+y}\right), \quad\quad y\geq0.
	\] It is a standard fact, and can be checked by a covariance calculation, that $Y$ is a standard planar Brownian motion with $Y(0)=0$. Let $\kappa$ be the first $y$ such that $Y(y)\in\partial N_q$. Now define
	\[Z(y) = (Y(y)-qh)^{3/2},\quad\quad y\geq0,
	\] where we raise an element of $\R^2$ to the $3/2$ power by
        thinking of it as an element of $\CC$. The map $z\mapsto
        z^{3/2}$ is a conformal isomorphism between the interior
          of $N$ and the interior of the right half plane $-i\HH$. Here we use the standard notation that $\HH=\{z\st \Im z\ge 0\}$ is the upper half plane and so $-i\HH = \{z\st \Re z\ge 0\}$ is the right half plane. Now $\kappa$ is the first $y$ such that $Z(y)\in \partial(-i\HH)=i\R$. By the conformal invariance of the planar Brownian motion, we know that $Z(\cdot)$ is a time changed planar Brownian motion starting at $Z(0)=(-qh)^{3/2}=(q|h|)^{3/2}$, so the probability density function of $Z(\kappa)$ is given by the Poisson kernel: for any $z\in\R$,
	\[
		\Prm[Z(\kappa)\in idz] = \frac{q^{3/2}|h|^{3/2}}{\pi(q^3|h|^3+z^2)}.
	\]
	Undoing the change of variables, we find that for all $\xi\ge 0$,
	\[
		\Prm[|Y(\kappa)-qh|\ge \xi] = \Prm[|Z(\kappa)|\geq\xi^{3/2}]=\frac2\pi \arctan\left(\frac{q^{3/2}|h|^{3/2}}{\xi^{3/2}}\right).
	\]
	Using the time-change $x=\frac{Ly}{L+y}$, which leads to $\frac{L+y}{L}=\frac{1}{1-x/L}$, we see that $\tau = \frac{L\kappa}{L+\kappa}$ is the first $x$ such that
	$V(x)/(1-x/L)$ hits the boundary of $N_q$, and moreover that $\frac{V(\tau)}{1-\tau/L} = Y(\kappa)$. This means that
	\begin{align*}
		\Prm[|V(\tau)|\ge q|h|+\xi] & \le \Prm[|V(\tau)|/(1-\tau/L)\ge q|h|+\xi]\notag \\&\le\Prm[|V(\tau)/(1-\tau/L)-qh|\ge \xi]\notag\\
		                            & = \Prm[|Y(\kappa)-qh|\ge \xi] \notag             \\&= \frac2\pi \arctan\left(\frac{q^{3/2}|h|^{3/2}}{\xi^{3/2}}\right),
	\end{align*}
	as claimed.
\end{proof}

\subsubsection{Proof of the upper bound $\sigma_L^2\leq CL^{-1/2}$}
We now have all of the necessary ingredients to prove Proposition~\ref{prop:ub}.
Recall that $\sigma_L^2=L\sum_{m=1}^M G_m$. (See  \eqref{eq:breakupsigmaL2}.) By Lemma~\ref{lem:GCSsym},   we need to control the factors on the r.h.s.\ of the following:
	\begin{equation}\label{e.261}
		G_m^2 \le 4e^{- q_{m-1}}Q_{q_m}(L,0)\Er[Q_{q_m}(L-\tau_m,V(\tau_m));\tau_m\le L/2].
	\end{equation}
	First we consider the case $1\le m\le M-1$.
	By our assumption of $1\leq q_m\leq L^{1/2}$ in \eqref{eq:bulkqbounds}, the inequality \eqref{eq:Qbd} tells us that
	\begin{equation}\label{eq:easybound}
		Q_{q_m}(L,0)\le CL^{-3/2}q_m^{3}e^{Cq_m^2/L}\le CL^{-3/2}q_m^{3}.
	\end{equation}
Similarly, we have
	\begin{equation}\label{eq:closebdpre}
		\begin{aligned}
			Q_{q_m}(L-\tau_m,V(\tau_m)) & \le C(L-\tau_m)^{-3/2}q_m^{3/2}e^{Cq_m(L-\tau_m)^{-1}(q_m+|V(\tau_m)|)} \\ & \hspace{6em}\times[q_m^{1/2}+|V(\tau_m)|^{1/2}][q_m+|\omega(V(\tau_m)|]. \end{aligned}
	\end{equation}
	On the event that $\tau_m\le L/2$, we have $L-\tau_m\ge L/2$, and by the fact that  \[
		\omega(V(\tau_m))=q_{m-1} (1-\tau_m/L)
	\]
	we have
	\[\omega(V(\tau_m))/q_m=q_{m-1}(1-\tau_m/L)/q_m\in [C^{-1},C], \quad\quad \text{ if } m\geq2,
	\]
	where $C>0$ is a universal constant. Also, we have $|V(\tau_m)|\ge C^{-1}\omega(V(\tau_m))$.
	Using these observations, or a simpler argument if $m=1$, in \eqref{eq:closebdpre}, we obtain
	\begin{equation}
		Q_{q_m}(L-\tau_m,V(\tau_m))\le CL^{-3/2}q_m^{5/2}[q_m^{1/2}+|V(\tau_m)|^{1/2}]e^{Cq_mL^{-1}|V(\tau_m)|}.\label{eq:closebd}
	\end{equation}
On the other hand, we also have by \eqref{eq:Qbddbyexp} that (still on the event $\tau_m\le L/2$)
	\begin{equation}
		Q_{q_m}(L-\tau_m,V(\tau_m))\le Ce^{C|V(\tau_m)|/L}.\label{eq:farbd}
	\end{equation}
	Then we can split up the expectation on the r.h.s.\ of \eqref{eq:Gm2bd} as (recalling that $q_m\leq L^\delta$ by \eqref{eq:bulkqbounds})
	\begin{align}
		\Er & [Q_{q_m}(L-\tau_m,V(\tau_m));\tau_m\le L/2]\notag                                                                              \\&=\Er[Q_{q_m}(L-\tau_m,V(\tau_m));\tau_m\le L/2\text{ and } |V(\tau_m)|\le L^{1-\delta}] \notag\\&\qquad+ \Er[Q_{q_m}(L-\tau_m,V(\tau_m));\tau_m\le L/2\text{ and } |V(\tau_m)|> L^{1-\delta}]\notag\\
		    & \le CL^{-3/2}q_m^{5/2}(q_m^{1/2}+\Er|V(\tau_m)|^{1/2}) + C\Er[e^{C|V(\tau_m)|/L};|V(\tau_m)|> L^{1-\delta}].\label{eq:splitup}
	\end{align}
	The first term in the last expression comes from \eqref{eq:closebd}, and the second term comes from \eqref{eq:farbd}.

	To deal with the first term in \eqref{eq:splitup}, we note that, by Lemma~\ref{lem:computeharmonicmeasure}, we have
	\begin{align}
		\Er|V(\tau_m)|^{1/2} & = \int_0^\infty \Prm[|V(\tau_m)|\ge \xi^2]d\xi
		\le q_m^{1/2}|h|^{1/2} +  \int_0^\infty \Prm[|V(\tau_m)|\ge q_m|h|+\xi^2]d\xi\notag                                              \\
		                     & \le q_m^{1/2}|h|^{1/2} +  \frac2\pi\int_0^\infty \arctan\left(\frac{q_m^{3/2}|h|^{3/2}}{\xi^3}\right)d\xi
		\le Cq_m^{1/2}.\label{eq:bulk}
	\end{align}
	For the second term in \eqref{eq:splitup}, note that for any $b\in [1,\infty)$ we have \begin{align*}
		\Er & [e^{b|V(\tau_m)|/L}; |V(\tau_m)|> L^{1-\delta}]                                                                            \\
		    & = \int_0^\infty \Prm[e^{b|V(\tau_m)|/L}\ge \xi\text{ and } |V(\tau_m)|> L^{1-\delta}]d\xi                                  \\
		    & \le \int_0^\infty \Prm[|V(\tau_m)|\ge \max\{ b^{-1}L\log\xi,L^{1-\delta}\}]d\xi                                            \\
		    & = e^{bL^{-\delta}}\Prm[|V(\tau_m)|\ge L^{1-\delta}]+\int_{e^{bL^{-\delta}}}^\infty \Prm[|V(\tau_m)|\ge b^{-1}L\log\xi]d\xi \\
		    & \le Ce^{bL^{-\delta}-L^{1-2\delta}/C}+C\int_{e^{bL^{-\delta}}}^\infty e^{-C^{-1}b^{-2}L(\log \xi)^2}d\xi,
	\end{align*}
	with the last inequality by two applications of \eqref{eq:crudeVtaumbd}.
 	The last integral we can estimate as \begin{align*}
		\int_{e^{bL^{-\delta}}}^\infty e^{-C^{-1}L(\log \xi)^2}d\xi =  \int_{bL^{-\delta}}^\infty e^{\eta-C^{-1}L\eta^2}d\eta & \le \int_{bL^{-\delta}}^\infty e^{(1-C^{-1}bL^{1-\delta})\eta}d\eta \\&= \frac{e^{(1-C^{-1}bL^{1-\delta})bL^{-\delta}}}{C^{-1}bL^{1-\delta}-1},
	\end{align*}
 	 provided that $L$ is large  enough that $1<C^{-1}bL^{1-\delta}$. (Recalling our assumption that $b\ge 1$, it is sufficient to assume that $L^{1-\delta}\ge C$.)
	Since $\delta<1/2$, this means that (taking $b$ to be a fixed constant $C$)
	\begin{equation}\label{eq:tailfinal}
		\Er[e^{C|V(\tau_m)|/L}; |V(\tau_m)|> L^{1-\delta}] \le Ce^{-C^{-1}L^{1-2\delta}}
	\end{equation}
	as long as $L$ is greater than some absolute constant.

	Now using \eqref{eq:bulk} and \eqref{eq:tailfinal} in \eqref{eq:splitup}, we obtain
	\begin{equation}
		\Er[Q_{q_m}(L-\tau_m,V(\tau_m));\tau_m\le L/2]\le CL^{-3/2}q_m^3+Ce^{-C^{-1}L^{1-2\delta}}\le CL^{-3/2}q_m^3.\label{eq:hardbound}
	\end{equation}
	Further using \eqref{eq:easybound} and \eqref{eq:hardbound} in \eqref{e.261}, we obtain, for $1\le m\le M-1$, that
	\begin{align}
		G_m^2 \le Ce^{- q_{m-1}}L^{-3}q_m^6.\label{eq:Gmbdfinal}
	\end{align}

	We still need to bound $G_M$. For this we use \eqref{e.261} again, but only with the bound \eqref{eq:Qbddbyexp} (eschewing \eqref{eq:Qbd} which is not useful for $q=\infty$). The resulting bound is
	\begin{align}
		G_M^2 & \le 4e^{- q_{M-1}}Q_{\infty}(L,0)\Er[Q_{\infty}(L-\tau_M,V(\tau_M));\tau_M\le L/2]\notag \\
		      & \le Ce^{- q_{M-1}}\Er[e^{CL^{-1}|V(\tau_M)|}]\le Ce^{- q_{M-1}}\label{eq:GMbd},
	\end{align}
	where the last step comes from another application of \eqref{eq:crudeVtaumbd}.

	Now we use \eqref{eq:Gmbdfinal} and \eqref{eq:GMbd} in \eqref{eq:breakupsigmaL2}, and recall \eqref{eq:qmdef} and \eqref{eq:bulkqbounds}, to obtain
	\begin{align*}
		\sigma_L^2 & \le CL^{-1/2}\sum_{m=1}^{M-1} e^{-\frac12 q_{m-1}}q_m^3 + CLe^{-\frac12q_{M-1}}\le CL^{-1/2}\left[\sum_{m=1}^{\infty} e^{-\frac12 (\xi^{m-1}_0-1)}(\xi^m_0-1)^3 + 1\right].
	\end{align*}
	The infinite series is summable, so this implies that $\sigma_L^2 \le CL^{-1/2}$, as claimed.\qed

\subsection{Lower bound}
We now complete the proof of  Theorem~\ref{prop:sigmaprop} by proving the lower bound. We will use the framework, notations, and estimates established in Sections~\ref{subsec:transinv}--\ref{subsec:bminwedge}. On the other hand, this section is independent of the proof of the upper bound in Section~\ref{subsec:ub}, and we will use some different notations from what is used there.
\begin{proposition}\label{prop:lb}
	There exists a constant $C>0$ so that, for all $L\ge 1$, we have $\sigma_L^2\ge C^{-1}L^{-1/2}$.
\end{proposition}

Fix a smooth, symmetric function  $\chi$ so that $\chi|_{[-1,1]}\equiv 0$ and $\chi|_{[-2,2]^\mathrm{c}}\equiv 1$. Define $\chi_L(x) = \chi(x)\chi(x-L/2)\chi(x-L)$.
For any $\kappa>0$, define
\begin{equation}
	q_{\kappa,L}(y)=\kappa\int_0^y\chi_L(x)\sgn(L/2-x)[x\wedge(L-x)]^{\kappa-1}dx,\qquad y\in [0,L].\label{eq:qdef}
\end{equation} (See Figure~\ref{fig:qklfig}.)
\begin{figure}
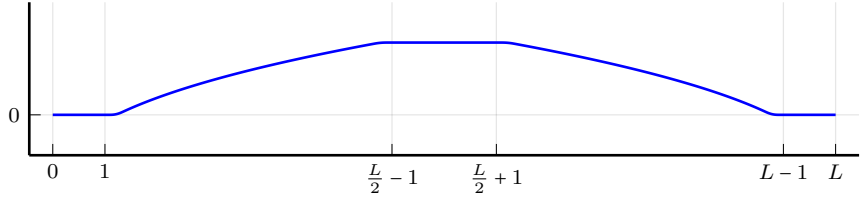


 	\caption{The function $q_{\kappa,L}$ ($\kappa = 0.4$).}\label{fig:qklfig}
\end{figure}
It is clear that   $q_{\kappa,L}(\cdot)$ is smooth, $q_{\kappa,L}(L/2+y)=q_{\kappa,L}(L/2-y)$ for all $y\in [0,L/2]$, and that
\begin{equation}\label{e.271}
	q_{\kappa,L}(y) \leq \kappa\int_0^y x^{\kappa-1}dx = y^{\kappa}\quad\quad \text{ for } y\in [0,L/2].
\end{equation}

Now fix $\gamma\in (0,1/2)$ and define the event
\[
	F=\left\{ \omega(V(x))\le 1-q_{\gamma,L}(x)\text{ for all }x\in [0,L]\right\}.
\]
Note that $F\subset E_{1,0,L}$.
The key to our proof of the lower bound will be the following ``entropic repulsion'' lemma, which states that conditional on $\omega(V(x))$ staying below $1$, it in fact has a positive probability of staying below $1-q_{\gamma,L}(x)$.
\begin{lemma}\label{lem:Fhappens}
	There exists a constant $C>0$ such that whenever $L\ge 1$ we have $\Prm[F\mid E_{1,0,L}]\ge C^{-1}$.
\end{lemma}

Before proving Lemma~\ref{lem:Fhappens}, we show how to use it to prove Proposition~\ref{prop:lb}.

\begin{proof}[Proof of Proposition~\ref{prop:lb}]
	On the event $F$, we have \begin{equation*}
		\prod_{k=1}^2 \left(\int_0^L e^{v_k\cdot V(x)}dx\right)^{-1} \ge  \left(\int_0^L e^{1-q_{\gamma,L}(x)}dx\right)^{-2}\ge C^{-1},
	\end{equation*}
	where $C$ is a positive constant independent of $L$ and the last step comes from \eqref{e.271}. Therefore, we have
	\begin{equation}\label{eq:sigmapreplb}
		\sigma_L^2 = \Er\left[\prod_{k=1}^2 \left(\int_0^L e^{v_k\cdot V(x)}dx\right)^{-1}\right]\ge C^{-1}\Prm_{0,0}^{L,0}[F]=C^{-1}\Prm_{0,0}^{L,0}[F\mid E_{1,0,L}]\Prm_{0,0}^{L,0}[E_{1,0,L}].\end{equation}
	By \eqref{eq:translateit} and \eqref{eq:ptNlowerbound}, we have
	\begin{equation}\label{eq:PE10L}
		\Prm_{0,0}^{L,0}[E_{1,0,L}]=\frac{p^N_L(-h,-h)}{p_L(0,0)}=2\pi
                L p^N_L(-h,-h)\ge C^{-1}L^{-3/2}.
	\end{equation}
	The claimed lower bound follows from applying Lemma~\ref{lem:Fhappens} and \eqref{eq:PE10L} in \eqref{eq:sigmapreplb}.
\end{proof}

Now we prove Lemma~\ref{lem:Fhappens}.
We follow the procedure of \cite{Bra83} as explicated in the proof
of  \cite[Proposition~2.1]{CHL19supp}\nocite{CHL19}. The works \cite{Bra83,CHL19supp} considered a one-dimensional Brownian bridge conditioned to stay above a line. The main ingredient is the Cameron--Martin theorem. Due to the two-dimensional nature of our problem, our proof requires an estimate for the expectation of the modulus of a two-dimensional Brownian bridge conditioned to stay in the cone; see Lemma~\ref{lem:expbd} below.

\begin{proof}[Proof of Lemma~\ref{lem:Fhappens}]
	In this proof we assume that our probability space is the space of $\R^2$-valued continuous
          functions on $[0,L]$, and the measure $\Prm$ is the
        standard two-dimensional Brownian bridge measure on this
        space, defined on the Borel $\sigma$-algebra, with $\Er$ the associated expectation.
	For a continuous function $Y:[0,L]\to\R$, we define \[
		g(Y)=-\int_0^L q_{\gamma,L}''(x) Y(x)dx.
	\]
	Now we have, using \eqref{eq:hshift}, that
	\begin{align}
		\Prm[F] & =\Prm[(V+q_{\gamma,L}h)\in E_{1,0,L}]\nonumber                                                      \\
		        & \ge\Prm[(V+q_{\gamma,L}h)\in E_{1,0,L}\text{ and }g(V_{1}+q_{\gamma,L}h_1)\le M],\label{eq:PF}
	\end{align}
	for any $M<\infty$. Recall that $h=(-\sqrt{6}/3,0)$. Here
        $V_1,h_1$ are the first components of $V$ and $h$, respectively.   We have by the Cameron--Martin theorem (see e.g.\ \cite[Theorem 2.6]{NRBSS11})
	that
	\begin{equation}
		\frac{d\Law(V+q_{\gamma,L}h)}{d\Prm}=\exp\left\{ h_1\int_{0}^{L}q_{\gamma,L}'(x)d V_1(x)-\frac{h_1^2}{2}\int_{0}^{L}|q_{\gamma,L}'(x)|^{2}dx\right\}.\label{eq:RNderiv}
	\end{equation}
	From \eqref{eq:PF} and \eqref{eq:RNderiv} we have
	\begin{equation}
		\begin{aligned}
			\Prm[F] & \ge\exp\left\{ -\frac{h_1^2}{2}\int_{0}^{L}|q_{\gamma,L}'(x)|^{2}dx\right\} \\&\qquad\times \Er\left[\exp\left\{ h_1\int_{0}^{L}q_{\gamma,L}'(x) d V_1(x)\right\} ;E_{1,0,L}\cap\{g(V_{1})\le M\}\right].\label{eq:applyRN}
		\end{aligned}
	\end{equation}
	We see from from the definition \eqref{eq:qdef} that
	\begin{equation}
		\int_{0}^{L}|q_{\gamma,L}'(x)|^{2}d x
		\le
		2\gamma^2\int_1^\infty x^{2(\gamma-1)} dx \leq C,\label{eq:squaredterm}
	\end{equation}
by our assumption of $\gamma<1/2$. Moreover, by using
        integration by parts (and the fact that $q_{\gamma,L}$ is
        smooth and supported on a compact subset of the interior of
        $[0,L]$ by definition),
we have
	\begin{equation}
		h_1\int_{0}^{L}q_{\gamma,L}'(x) d V_1(x) =-h_1\int_{0}^{L}q_{\gamma,L}''(x)V_1(x)d x=-\frac{\sqrt6}{3}g(V_1).\label{eq:bdintermsofgV1}
	\end{equation}
	Using \eqref{eq:squaredterm} and \eqref{eq:bdintermsofgV1} in \eqref{eq:applyRN}, we obtain
	\begin{equation}
		\Prm[F]\ge c(M)\Prm\left[E_{1,0,L}\cap\{g(V_{1})\le M\}\right],\label{eq:PFgecEg}
	\end{equation}
	for a constant $c(M)>0$ depending on $M$ but not on $L$.

	Now since $F\subset E_{1,0,L}$, we can divide through by $\Prm[E_{1,0,L}]$ in \eqref{eq:PFgecEg} to obtain
	\begin{equation}
		\Prm[F\mid E_{1,0,L}] \ge c(M)\Prm[g(V_{1})\le M\mid E_{1,0,L}].\label{eq:justneedtodealwithg}\end{equation}
	So it remains to estimate the last conditional probability. By Lemma~\ref{lem:expbd} below, we have whenever $x\in [1,L-1]$ that
	\begin{equation*}
		\Er[|V_{1}(x)|\mid E_{1,0,L}]\le C[x\wedge (L-x)]^{1/2},
	\end{equation*}
	and so by linearity of conditional expectation and the fact that $q_L''$ is supported on $[1,L-1]$, we have
	\begin{equation}
		\Er\left[\int_0^{L}|q_{\gamma,L}''(x)V_{1}(x)|dx\ \middle|\ E_{1,0,L}\right]\le C\int_1^{L-1} |q_{\gamma,L}''(x)|[x\wedge (L-x)]^{1/2}dx\le C,\label{eq:boundqprimeprimebyC}
	\end{equation}
	for a constant $C<\infty$ independent of $L$. The last inequality in \eqref{eq:boundqprimeprimebyC} is derived as follows. By symmetry and the assumption on $q_{\gamma,L}$, we only need to consider the integration domain of $x\in [1,\tfrac{L}{2}-1]$. Recalling \eqref{eq:qdef}, we have $q_{\gamma,L}''(x)=\gamma(\gamma-1)x^{\gamma-2}$ in $x\in [2,\tfrac{L}{2}-2]$, and since $\gamma-2+\tfrac12<-1$, the corresponding integral is bounded independent of $L$. For $x\in [1,2]\cup [\tfrac{L}{2}-2,\tfrac{L}{2}-1]$, we also have $|q_{\gamma,L}''(x)|x^{1/2}\leq C $, and the last inequality in \eqref{eq:boundqprimeprimebyC} follows.

	By the definition of $g$, the inequality \eqref{eq:boundqprimeprimebyC}, and Markov's inequality, we obtain
	\begin{equation}\label{eq:applyMarkov}
		\Prm[g(V_1)\ge M\mid E_{1,0,L}]\le \Prm\left[\int_0^{L}|q_{\gamma,L}''(x)V_{1}(x)|dx\ge M\ \middle|\ E_{1,0,L}\right]\le C/M.
	\end{equation}
	Substituting \eqref{eq:applyMarkov} into \eqref{eq:justneedtodealwithg} and taking $M$ sufficiently large (independent of $L$),
	we obtain
	\[
		\Prm[F\mid E_{1,0,L}] \ge \frac12 c(M).
	\]
	This completes the proof.
\end{proof}

It remains to prove the lemma we used in the last proof.

\begin{lemma} \label{lem:expbd} There exits a constant $C>0$, independent of $L$, so that whenever $L\ge 2$ and $x\in [1,L-1]$ we have
	\begin{equation}\label{eq:expbd}
		\Er[|V(x)|\mid E_{1,0,L}]\le C[x\wedge (L-x)]^{1/2}.
	\end{equation}
\end{lemma}
\begin{proof}We note that (using \eqref{eq:translateit}) \begin{equation}
		\Er[|V(x)|\mid E_{1,0,L}]= \frac{\int_{N} |u+h| \,p_x^N (-h,u)p_{L-x}^N(u,-h)du}{p_L^N(-h,-h)}.\label{eq:probtoobigint}
	\end{equation}
	For the denominator we use \eqref{eq:ptNlowerbound} to write
	\begin{equation}
		p_L^N(-h,-h)\ge C^{-1}L^{-5/2}.\label{eq:denombd}
	\end{equation}
	To investigate the numerator of \eqref{eq:probtoobigint}, we note that
	\begin{equation}\label{e.272}
		\int_{N} |u+h| \,p_x^N (-h,u)p_{L-x}^N(u,-h)du \leq |h|\,p_L^N(-h,-h)+S,
	\end{equation}
	with
	\begin{align*}
		S & \coloneqq\int_{N} |u| p_x^N (-h,u)p_{L-x}^N(u,-h)du\notag                         \\&=\int_0^\infty\int_{-\pi/3}^{\pi/3} r^2 p_x^N(-h,re^{i\theta})p_{L-x}^N(re^{i\theta},-h)d\theta dr\notag\\
		  & \le \frac{Ce^{-\frac13[x^{-1}+(L-x)^{-1}]}}{[x(L-x)]^{5/2}}\int_0^\infty r^5
		e^{-\frac12r^2[x^{-1}+(L-x)^{-1}]+Cr[x^{-1}+(L-x)^{-1}]}dr
	\end{align*}
	by \eqref{eq:ptxybd}.
	Without loss of generality, we consider the case of $x\le L-x$. By the change of variables $r\to rx^{1/2}$, we obtain
	\begin{align}
		S & \le \frac{Cx^{1/2}}{(L-x)^{5/2}}\int_0^\infty r^5
		e^{-\frac12r^2x[x^{-1}+(L-x)^{-1}]+Crx^{1/2}[x^{-1}+(L-x)^{-1}]}dr\notag \\
		  & \le \frac{Cx^{1/2}}{L^{5/2}}\int_0^\infty r^5
		e^{-\frac12r^2+Crx^{-1/2}}dr\notag                                       \\
		  & \le \frac{Cx^{1/2}e^{Cx^{-1}}}{L^{5/2}}\int_0^\infty r^5
		e^{-\frac14r^2}dr \leq Cx^{1/2}L^{-5/2},\label{eq:chgvar}
	\end{align}
	where
	we used the assumption that $x\ge 1$ and  Young's inequality. Using \eqref{eq:denombd}, \eqref{e.272} and \eqref{eq:chgvar} in \eqref{eq:probtoobigint} gives us \eqref{eq:expbd}.
\end{proof}

\section{Further discussions}
\label{s.discussion}

\subsection{Another proof of CLT}

The Clark-Ocone representation suggests another way of proving the central limit theorem, which says that for fixed $L>0$,
\begin{equation}\label{e.clt1}
	\frac{\log \U(t,x)+\gamma_L t}{\sqrt{t}}\Rightarrow N(0,\sigma_L^2), \quad\quad \mbox{ as }t\to\infty.
\end{equation}
We sketch a heuristic argument here. Throughout the section, $L>0$ is fixed (so the estimates and proportionality constants depend on $L$).

In \eqref{e.co}, $f,g$ can be chosen as arbitrary measures, and by  choosing $f=\delta_x$ and $g=\nu$ (with $\nu$ an arbitrary probability measure on $\bT_L$) we have
\begin{equation}\log \U(t,x)-\E \log \U(t,x)=\int_0^t\int_0^L \mathscr{G}_{t,s}(f,\rho(s;g),y)\rho(s,y;g)\eta(s,y)dyds.
\end{equation}
Here $\U$ is the solution to SHE starting from $g$.

Since $L$ is fixed, by \cite[Theorem 2.3]{GK21}, we know that
the process $\{\rho(t;g)\}_{t\geq0}$ mixes exponentially fast.
This would   imply, see \cite[(2.24)]{GK21},
\begin{equation}\label{e.EEh}
	\E \log \U(t,x)=-\gamma_L t+O(1).
\end{equation}

With $f=\delta_x$, we have
\[
	\begin{aligned}
		\mathscr{G}_{t,s}(f,h,y) &= \E\bigg[\frac{\int_0^L f(x') \cZ(t,x';s,y)dx' }{\int_0^L\int_0^L  f(x')\cZ(t,x';s,y')h(y')dx'dy'}\bigg]                                                         \\
		                         &= \E\bigg[\frac{ \cZ(t,x;s,y) }{\int_0^L  \cZ(t,x;s,y')h(y')dy'}\bigg]=\E \bigg[ \frac{\tilde{\rho}_{t,x}(s,y)}{\int_0^L \tilde{\rho}_{t,x}(s,y')h(y')dy'}\bigg],
	\end{aligned}
\]
where
\[
	\tilde{\rho}_{t,x}(s,y):= \frac{\cZ(t,x;s,y)}{\int_0^L \cZ(t,x;s,y')dy'}.
\]
It is straightforward to check that for each $t>0$ and $x\in \bT_L$ fixed,
\[
	\left(\tilde{\rho}_{t,x}(s,\cdot)\right)_{s\in[0,t]}\stackrel{\text{law}}{=}\left(\rho(t-s,\cdot\,; \delta_x) \right)_{s\in[0,t]}
\]
where the r.h.s.\ is the endpoint distribution of the directed polymer
of length $t-s$, starting from $x$. Thus, by the exponential
mixing of $\rho(t,\cdot;\delta_x)$,  we expect the following approximation
\[
	\mathscr{G}_{t,s}(f,h,y) \approx \E_B \bigg[\frac{e^{B(y)} }{\int_0^L e^{B(y')}h(y')dy'}\bigg]=\mathcal{A}(h,y), \quad\quad \mbox{ if } t-s\gg1;
\]
see \eqref{e.EfG}.
Recall that when $f$ is sampled from the invariant distribution $\mu$,  the above approximation is exact, as in \eqref{e.EfG}.

With the above approximation, we expect that, for any $\beta\in(0,1)$,
\begin{equation}\label{e.9271}
	\begin{aligned}
		\log \U(t,x)-\E \log \U(t,x) & =\int_0^t\int_0^L \mathscr{G}_{t,s}(f,\rho(s;g),y)\rho(s,y;g)\eta(s,y)dyds                   \\
		                             & \approx  \int_0^{t-t^\beta}\int_0^L \mathscr{G}_{t,s}(f,\rho(s;g),y)\rho(s,y;g)\eta(s,y)dyds \\
		                             & \approx \int_0^{t-t^\beta} \int_0^L \mathcal{A}(\rho(s;g),y)\rho(s,y;g)\eta(s,y)dyds         \\
		                             & \approx \int_0^t \int_0^L \mathcal{A}(\rho(s;g),y)\rho(s,y;g)\eta(s,y)dyds=: M_t.
	\end{aligned}
\end{equation}
The above ``$\approx$''s mean that the $L^2$ norm of the error is of order $o(\sqrt{t})$. The $M_t$ defined above is a martingale. An application of the martingale central limit theorem, as in \cite[Lemma 5.29]{GK21}, implies that
\[
	\frac{\log \U(t,x)-\E \log \U(t,x)}{\sqrt{t}}\Rightarrow N(0,\sigma_L^2), \quad\quad \mbox{ as }t\to\infty.
\]
Combining with \eqref{e.EEh}, this completes the proof of \eqref{e.clt1}.

\subsection{Connection to Liouville quantum mechanics}\label{subsec:LQM}
Recall that from \eqref{e.defsigma} and Lemma~\ref{l.BBstationarity}, we have
\[
	\sigma_L^2=L\E \frac{1}{\int_0^L e^{B_1(x)+B_3(x)}dx\int_0^L e^{B_2(x)+B_3(x)}dx},
\]
where $B_j$ are independent Brownian bridges with $B_j(0)=B_j(L)=0$, $j=1,2,3$. Random variables of the form $\int_0^L e^{B(x)}dx$ have been studied extensively due to the applications in physics and mathematical finance, where $B$ could be a Brownian motion or a Brownian bridge. The densities of these random variables are known; see the survey \cite{MY05}. One way to derive the density of such a random variable is through the connection to the Schr\"odinger operator with the so-called Liouville potential $-\frac12\partial_x^2 +e^{x}$, whose Green's function can be written down explicitly,
see \cite[Theorem 4.1]{MY05} and its proof. It suggests that, for our problem, to study the asymptotics of $\sigma_L^2$, another possible approach is to try to derive the density of the joint distribution of $\int_0^L e^{W_1(x)}dx$ and $\int_0^L e^{W_2(x)}dx$, with $W_1,W_2$ two correlated Brownian bridges. The corresponding Schr\"odinger operator would take the form
\[
	-\partial_{x_1}^2-\partial_{x_1x_2}-\partial_{x_2}^2+e^{x_1}+e^{x_2},
\]
which, after a change of variables, becomes a generalized Toda lattice. It would be interesting to see whether there are integrable tools that could lead to a good understanding of the spectral properties of the above operator, using which we may compute the value of $\sigma_L^2$ explicitly. Similar approaches appeared in e.g.\ \cite{BO11,OCo12} to study the distributions of Brownian functionals. It is also worth mentioning that, by the replica method and the study of the ground state energy of the delta-Bose gas on a torus, there is actually a prediction of the value of $\sigma_L^2$, as an explicit integral involving the hyperbolic tangent function; see \cite[Equation (33)]{BD00}.

\subsection{Noise with general covariance function in arbitrary dimensions}
\label{s.highd}

The approach developed in this paper suggests a possible way of understanding the universality.  The noise we considered  is the $1+1$ spacetime white noise in which case  the invariant measure for the   KPZ equation is explicit, given by the Brownian bridge. Suppose we consider a noise that is white in time and colored in space, with a spatial covariance function $R(\cdot)$ that is independent of $L$:
\[
	\E\, \eta(t,x)\eta(s,y)=\delta(t-s)R(x-y).
\]
As long as $R(\cdot)$ decays fast enough, one expects a similar result
as  \eqref{e.conjecture1} in $d=1$. Denote the solution to the KPZ
equation by $h_L$ and assume that  $h_L(0,x)-h_L(0,0)$  is sampled from the unique    invariant measure
$\tilde{\pi}$, for which we do not have the explicit formula.  By following
the same proof verbatim as presented in Sections \ref{s.variance} and \ref{s.decomposition}, one arrives at
\begin{equation}\label{e.decom2}
	h_L(t,0)-\E h_L(t,0)=I_L(t)+(Y_L(0)-Y_L(t)),
\end{equation}
with $\Var \,I_L(t)=t \tilde{\sigma}_L^2$ and $\{Y_L(t)\}_{t\geq0}$ a stationary process.  The variance takes the form
\begin{equation}\label{e.iv1}
	\tilde{\sigma}_L^2=\int_0^L \int_0^L \E\, \tilde{\rho}(x_1)\tilde{\rho}(x_2)R(x_1-x_2)dx_1dx_2, \end{equation}
where
\[
	\tilde{\rho}(x)=\E\bigg[\frac{e^{H_1(x)+H_2(x)}}{\int_0^L e^{H_1(x')+H_2(x')}dx'}\,\bigg|\,H_1\bigg],
\]
with $H_1$ and $H_2$  sampled independently from  $\tilde{\pi}$. In addition, for each $t\geq0$,
\begin{equation}\label{e.iv2}
	Y_L(t)\stackrel{\text{law}}{=} \E\bigg[\log \int_0^L e^{H_1(x)+H_2(x)}dx\,\bigg|\,H_1\bigg].
\end{equation}
A natural conjecture is that, in $d=1$ and with $R(\cdot)$ decaying sufficiently fast,  we have $\tilde{\sigma}_L^2 \asymp L^{-1/2}$ and $\Var\,
	Y_L(t)\asymp L$. This would imply \eqref{e.conjecture1} for the noise with a general spatial covariance function. In some sense, a
contribution of this paper is to reduce the ``dynamic'' problem of
estimating the variance of the height function, in the super-relaxation and  relaxation regime,
to the ``static'' problem of studying the  integrals \eqref{e.iv1} and
\eqref{e.iv2}, which only involve the invariant measure. In the case
of
$1+1$ spacetime white noise, $H_j$ are  independent copies of Brownian
bridges for which we did some explicit calculations. The decomposition \eqref{e.decom2} and the identities \eqref{e.iv1} and \eqref{e.iv2} hold in high dimensions as well. It is unclear what would be the picture there, although another natural conjecture would be that the critical length scale in high dimensions still comes from  a balance of the two terms  $t\tilde{\sigma}_L^2\sim \Var Y_L(t)$.

\appendix

\section{Square integrability of the Malliavin derivative}\label{s.sqin}
Fix $t>0$ and deterministic $f,g\in C_+(\bT_L)$. Recall that 
\[
\widetilde{\rho}_{t,f}(s,y)=\frac{\int_0^L \cZ(t,x;s,y)f(x)dx}{\int_0^L\int_0^L \cZ(t,x;s,y')f(x)dxdy'}, \quad s<t,
\]
and
\[
\rho(s,y;g)=\frac{\int_0^L \cZ(s,y;0,z)g(z)dz}{\int_0^L\int_0^L \cZ(s,y';0,z)g(z)dzdy'},\quad s>0.
\]
By \eqref{e.DsyX}, we have  
\[
\D_{s,y}X_{f,g}(t)=\frac{\widetilde{\rho}_{t,f}(s,y)\rho(s,y;g)}{\int_0^L\widetilde{\rho}_{t,f}(s,y')\rho(s,y';g)dy'}.
\]
The goal is to show the following lemma.
\begin{lemma}
We have 
$\int_0^t\int_0^L \E |\D_{s,y}X_{f,g}(t)|^2 dyds<\infty$.
\end{lemma}
\begin{proof}
 Without loss of generality, we assume that $\int f=\int g=1$. Note that $\widetilde{\rho}_{t,f}(s,\cdot)$ is the density of 
the ``backward'' polymer of length $t-s$ with its starting point
distributed according to $f(\cdot)$, and $\rho(s,\cdot\,;g)$ is the
density of the ``forward'' polymer of length $s$ with the starting
point distributed according to $g(\cdot)$. We  write 
\[
\int_0^t\int_0^L \E |\D_{s,y}X_{f,g}(t)|^2 dyds=\left(\int_0^{t/2}+\int_{t/2}^t\right)\int_0^L \E |\D_{s,y}X_{f,g}(t)|^2 dyds, 
\]
and consider first  the integration over $s\in(0,t/2)$. To deal
with the integration over $s\in(t/2,t)$ we just need to switch
the roles of $f$ and $g$ in the argument below. In the ensuing
computations, the constant $C$ may depend on $t,f,g$.

By \cite[Lemma B.1]{GK21}, we have for any $p>1$ that
\begin{equation}\label{e.mmbd}
\sup_{s\in (0,t/2),y\in [0,L]} \E \widetilde{\rho}_{t,f}(s,y)^p+\E \widetilde{\rho}_{t,f}(s,y)^{-p}<\infty.
\end{equation}
Thus, for $s\in(0,t/2)$, we can bound 
\[
\begin{aligned}
\E |\D_{s,y}X_{f,g}(t)|^2&\leq \E\bigg[\widetilde{\rho}_{t,f}^2(s,y)\rho^2(s,y;g)\int_0^L \widetilde{\rho}_{t,f}(s,y')^{-2}\rho(s,y';g) dy'\bigg]\\
&\leq C \E \bigg[\rho^2(s,y;g) \int_0^L \rho(s,y';g)dy'\bigg]=C\E \rho^2(s,y;g).
\end{aligned}
\]
Here the first ``$\leq$'' comes from Jensen's inequality
$$
\left(\int_0^L \widetilde{\rho}_{t,f}(s,y') \rho(s,y';g)
  dy'\right)^{-2} \le \int_0^L \widetilde{\rho}_{t,f}(s,y')^{-2}\rho(s,y';g) dy';
$$  the second ``$\leq$'' comes from the fact that
$\widetilde{\rho}_{t,f}(s,\cdot)$ and $\rho(s,\cdot;g)$ are
independent, and we first average $\widetilde{\rho}$ then apply
\eqref{e.mmbd} together with the H\"older inequality. The last ``='' is because $\rho(s,\cdot;g)$ is a probability density. Therefore, we have 
\begin{equation}\label{e.301}
\int_0^{t/2}\int_0^L \E |\D_{s,y}X_{f,g}(t)|^2 dyds \leq C\int_0^{t/2} \int_0^L\E \rho^2(s,y;g) dyds.
\end{equation}
By \cite[Lemmas B.1, B.2]{GK21}, we have (remembering that $g$ is a density)
\[
\begin{aligned}
\E \rho^2(s,y;g) &\leq C \sqrt{\E (\int_0^L \cZ(s,y;0,z)g(z)dz)^4}\\
&\leq C\left(\int_0^L \big(\E\cZ^4(s,y;0,z)\big)^{1/4} g(z)dz\right)^2\\
&\leq C \left(\int_0^L (1+p_s(y-z)p_s(0))^{1/2} g(z)dz\right)^2\\
& \leq C \int_0^L (1+p_s(y-z)p_s(0))  g(z)dz \\
&\leq  C\left(1+s^{-1/2}\|g\|_\infty\right).
\end{aligned}
\]
Here 
\begin{equation}\label{e.heatk1}
p_t(x)=\sum_{k\in \Z} \frac{1}{\sqrt{2\pi t}} e^{-\frac{|x+kL|^2}{2t}}
\end{equation} is the Green's function of the heat equation on
$\bT_L$, and for the last ``$\leq$'' we used the fact that $p_s$ is a
density and $p_s(0)\leq
Cs^{-1/2}$. Plugging into \eqref{e.301}, the proof is complete.
\end{proof}

\section{Proof of \eqref{e.reversal}}
 The goal is to show that, for each fixed $t>0$, we have 
\begin{equation}\label{e.reversal1}
\{ \cZ(t,x;0,y)\}_{x,y\in\bT_L}\stackrel{\text{law}}{=}\{ \cZ(t,y;0,x)\}_{x,y\in\bT_L}.
\end{equation}
Recall by \eqref{e.propa}, that $\cZ(t,x;0,y)$ solves 
	\[
	\begin{aligned}
		 & \partial_t \cZ(t,x;0,y)=\frac12\Delta_x \cZ(t,x;0,y)+\cZ(t,x;0,y) \eta(t,x), \quad\quad t>0, \\
		 & \cZ(0,x;0,y)=\delta(x-y).
	\end{aligned}
	\]
	We will prove \tk{\eqref{e.reversal1}} through an approximation. 
Let $\eta_\kappa(t,x)$, $\kappa>0$ be a spatial mollification of
$\eta(t,x)$ obtained as follows: take a mollifier $\phi_\kappa$
on $\bT_L$, such that
$\phi_\kappa(\cdot)\to\delta(\cdot)$, as
$\kappa\to0$, and define $\eta_\kappa(t,x)=\int_0^L
\phi_\kappa(x-y)\eta(t,y)dy$. Let $\cZ_\kappa$ solve the above
equation with $\eta$ replaced by the spatially smooth noise  $\eta_\kappa$. Then by the Feynman-Kac formula \cite[Theorem 2.2]{BC95} (while \cite{BC95} is on the equation on $\R$, the same result holds on $\bT_L$), we have
 \[
 \cZ_\kappa(t,x;0,y)=p_t(x-y)\E[e^{\int_0^t\eta_\kappa(t-s,B_s)ds-\frac12R_\kappa(0)t} \mid B_0=x,B_t=y].
 \]
 Here $p_t(\cdot)$ was defined in \eqref{e.heatk1}, $\E[\,\cdot\,|\, B_0=x,B_t=y]$ is the expectation on the Brownian bridge on the torus $\bT_L$, and the constant $R_\kappa(0)=\int_0^L \phi_\kappa(x)^2 dx$. Since 
 \[
 \{\eta_\kappa(s,y)\}_{s\in[0,t],y\in \bT_L}\stackrel{\text{law}}{=}\{\eta_\kappa(t-s,y)\}_{s\in[0,t],y\in \bT_L},
 \]
 we have 
 \[
 \cZ_{\kappa}(t,x;0,y)\stackrel{\text{law}}{=} p_t(x-y)\E[e^{\int_0^t\eta_\kappa(s,B_s)ds-\frac12R_\kappa(0)t} \mid B_0=x,B_t=y].
 \]
 By the change of variable $s'=t-s$ and the time reversal property of the Brownian bridge, we further have the r.h.s. equals to
 \[
 \begin{aligned}
 &p_t(x-y)\E[e^{\int_0^t\eta_\kappa(t-s,B_{t-s})ds-\frac12R_\kappa(0)t} |B_0=x,B_t=y]\\
 &=p_t(x-y)\E[e^{\int_0^t\eta_\kappa(t-s,B_{s})ds-\frac12R_\kappa(0)t} |B_0=y,B_t=x]=\cZ_{\kappa}(t,y;0,x).
 \end{aligned}
 \]
 Sending  $\kappa\to0$ and applying \cite[Theorem 2.2]{BC95} again, we derive \eqref{e.reversal1}.

\section{Estimates for modified Bessel functions}\label{appdx:bessel}
In this brief appendix we recall standard definitions and estimates, used in Section~\ref{sec.sigmaasympt}, on the modified Bessel functions of the first kind. We will use the formula
\begin{equation}
	I_{\nu}(z)=\frac{\left(z/2\right)^{\nu}}{\pi^{1/2}\Gamma(\nu+1/2)}\int_{0}^{\pi}e^{z\cos\eta}\sin^{2\nu}\eta d\eta\label{eq:Inuzintegralrep}
\end{equation}
 for these functions; see \cite[(9.8.18) on p. 376]{AS64}.

\begin{lemma}
	\label{lem:besselbd}For all $\nu,z\ge0$ we have
	\[
		\frac{\left(z/2\right)^{\nu} e^{-z}}{\Gamma(\nu+1/2)}\le I_{\nu}(z)\le\frac{\left(z/2\right)^{\nu} e^{z}}{\Gamma(\nu+1/2)}.
	\]
\end{lemma}

\begin{proof}
	We first note that, for all $\nu\ge0$, we have
	\begin{equation}
		\int_{0}^{\pi}\sin^{2\nu}\eta d\eta=\frac{\sqrt{\pi}\Gamma(\nu+1/2)}{\Gamma(\nu+1)}.\label{eq:sinbound}
	\end{equation}
	From \eqref{eq:Inuzintegralrep} we have for all $z\ge0$ that
	\begin{align*}
		I_{\nu}(z) & \in\left(\frac{\left(z/2\right)^{\nu}}{\pi^{1/2}\Gamma(\nu+1/2)}\int_{0}^{\pi}\sin^{2\nu}\eta d\eta\right)\cdot[ e^{-z}, e^{z}]=\frac{\left(z/2\right)^{\nu}}{\Gamma(\nu+1)}\cdot[ e^{-z}, e^{z}],
	\end{align*}
	which was the claim.
\end{proof}
\begin{lemma}
	\label{lem:besselsumbd}For any $\alpha>1$ there is a constant $C<\infty$ (depending on $\alpha$) so that, for
	all $q\ge0$, we have
	\begin{align}
		\sum_{j=1}^{\infty}j^{2}I_{\alpha j}(q) & \le C e^{q}\left[ e^{(q/2)^{\alpha}}-1\right]\le Cq^{\alpha} e^{Cq}.\label{eq:besselsumbd}
	\end{align}
\end{lemma}

\begin{proof}
	By Lemma~\ref{lem:besselbd}, we have for any $q\ge0$ that
	\begin{align*}
		\sum_{j=1}^{\infty}j^{2}I_{\alpha j}(q) & \le e^{q}\sum_{j=1}^{\infty}\frac{j^{2}(q/2)^{\alpha j}}{\Gamma(\alpha j+1/2)}. \end{align*}
	This sum evidently converges absolutely and uniformly on compact subsets of $[0,\infty)$, and so in particular we have a constant $C$ so that \begin{equation}\sum_{j=1}^{\infty}j^{2}I_{\alpha j}(q)\le Cq^\alpha\text{ for all $q\in [0,1]$.}\label{eq:sumforsmall}\end{equation} On the other hand, for $q\ge 1$, we have
	\begin{equation}
		\sum_{j=1}^{\infty}\frac{j^{2}(q/2)^{\alpha j}}{\Gamma(\alpha j+1/2)}\le C\sum_{j=1}^{\infty}\frac{j^{2}q^{\alpha j}}{\lfloor\alpha j-1/2\rfloor!}\le C\sum_{k=0}^{\infty}\frac{(k+1)^{2}q^{k+3/2}}{k!}\le Ce^{Cq}.\label{eq:sumforbig}
	\end{equation}
	Combining \eqref{eq:sumforsmall} and \eqref{eq:sumforbig}, we get \eqref{eq:besselsumbd}.
\end{proof}

\end{document}